\documentstyle{article}

\newtheorem{theorem}{Theorem}[subsection]

\newtheorem{lemma}[theorem]{Lemma}
\newtheorem{definition}[theorem]{Definition}
\newtheorem{corollary}[theorem]{Corollary}
\newtheorem{proposition}[theorem]{Proposition}
\newcommand{\qed}{\rule{1mm}{3mm}}     
\newenvironment{proof}{\vspace*{\parsep}\noindent {\bf proof:}}{\qed\\[1em]}
\begin{document}

\title{IKP and Friends}
\author{Robert S. Lubarsky\thanks{I would like to thank Anil Nerode for
bringing the subject of intuitionistic set theory to my attention, Michael Rathjen for
his encouragement in bringing this work to completion, and the referee for the careful
reading and insistence on including the details, even when I didn't think they were 
necessary.}
\\1755 NE 18th St.\\Ft. Lauderdale, FL 33305\\USA
\\Robert.Lubarsky@alum.mit.edu
\\Robert.Lubarsky@att.net}
\date{December 2001}
\maketitle
\section{Introduction}

There has been increasing interest in intuitionistic methods over the years.
Still, there has
been relatively little work on intuitionistic set theory, and most of that has
been on
intuitionistic ZF. This investigation is about intuitionistic admissibility and
theories of
similar strength.

There are several more particular goals for this paper. One is just to get some
more Kripke
models of various set theories out there. Those papers that have dealt with IZF
usually
were more proof-theoretic in nature, and did not provide models. Furthermore,
the
inspirations for many of the constructions here are classical forcing
arguments.
Although
the correspondence between the forcing and the Kripke constructions are not
made
tight,
the relationship between these two methods is of interest (see \cite{jim} for
instance) and
some examples, even if only suggestive, should help us better understand
the relationship between forcing and Kripke constructions. Along different
lines, the
subject of
least and greatest fixed points of inductive definitions, while of interest to
computer
scientists, has yet to be studied constructively, and probably holds some
surprises.
Admissibility is of course the proper set-theoretic context for this study.
Finally, while most
of the classical material referred to here has long been standard, some of it
has not been
well codified and may even be unknown, so along the way we'll even fill in
a gap
in the
classical literature.

The next section develops the basics of IKP, including some remarks on fixed
points of
inductive definitions. After that some classical theories related to KP are
presented, and
the question of which imply which others is completely characterized. While
they
are not
equivalent in general, when restricted to initial segments of (classical) L
they
are.
However, the section after that shows that in intuitionistic L this equivalence
breaks down.
We close with some questions.

\section{IKP} \label{secIKP}
The axioms of classical KP are: Empty Set, Pairing, Union, Extensionality,
Foundation (as
a schema for all definable classes), $\Delta_{0}$ Comprehension (also known as
$\Delta_{0}$ Separation), and $\Delta_{0}$ Bounding (also known as $\Delta_{0}$
Collection).
Often Infinity is adjoined; we will also use the axiomatisation with
Infinity in this paper. There is not much trouble adapting these to
an
intuitionistic setting. The concept of a $\Delta_{0}$ formula needs no change.
As usual,
Foundation must be replaced by $\in$-Induction: $\forall x (\forall y \in x \;
\phi (y) \rightarrow \phi (x)) \rightarrow \forall x \; \phi (x)$, where $\phi$
is a formula in the language of set theory. Also, the standard version of
Infinity, as
in \cite{powell} for instance ($\exists x \mbox{ }[(\exists y \; y \in x)
\wedge
(\forall y \in x \; \exists z \in x \; y \in z)]$), won't do, because you need
Power Set and $\Pi_{1}$
Comprehension to get $\omega$; instead, take the version of Infinity from
\cite{me} ($\exists x \mbox{ }[0 \in x \; \wedge \; \forall y \in x \;
\exists z
\in x (z = y \cup \{y\}) \; \wedge \; \forall y \in x (y = 0 \; \vee \; \exists
z \in x \mbox{ }y = z \cup \{z\})]$), which
axiomatizes $\omega$ straightforwardly. We take IKP to be KP with these few
adaptations.

This choice can be further justified. Two other axiomatizations of KP are to
replace
$\Delta_{0}$ Bounding with either $\Sigma$ Bounding or $\Sigma$ Reflection.
Taking the
$\Sigma$ formulas to be the closure of the $\Delta_{0}$ formulas under $\wedge,
\vee,
\exists$, and bounded quantification, these equivalences hold
intuitionistically, for the same reasons they do classically. Rather than
repeating these completely standard arguments here,
the reader is referred to \cite{barwise}, theorems I.4.3 and I.4.4, the proofs
of which go through unchanged in IKP. In addition, under IKP every $\Sigma$
formula is equivalent to a $\Sigma_{1}$ formula, proven inductively on
formulas.

One could also consider many of the standard consequences of KP. The ones
of interest to us are:
\begin{enumerate}
\item	$\Sigma$ Replacement,
\item	Strong $\Sigma$ Replacement,
\item	$\Delta$ Comprehension,
\item	$\Sigma$ Recursion,
\item the Second Recursion Theorem, and
\item	$\Sigma$ Inductive Definitions.
\end{enumerate}
$\Sigma$ Replacement and Strong $\Sigma$ Replacement ($\forall x \in A \;
\exists y \;
\phi(x, y) \rightarrow \exists f \; \forall x \in A \; f(x) \neq \emptyset
\wedge \forall y \in f(x) \;
\phi(x, y)$) hold intuitionistically, by the same proofs as in the classical
case; again, the reader is referred to \cite{barwise}, theorems I.4.6 and
I.4.7.
$\Delta$
Comprehension requires considerably more care. Usually a $\Delta$ property
is defined as being
given by a pair $\phi, \psi$ of $\Sigma$ and $\Pi$ formulas respectively, but
when it's used
the first thing you do is deal with $\neg \psi$ as a $\Sigma$ formula. This
won't work
intuitionistically. The next obvious guess at what a $\Delta$ property
should be is a pair
$\phi, \psi$ of
$\Sigma$ formulas such that $\forall x \; [(\phi(x) \vee \psi(x)) \; \wedge \;
\neg(\phi(x)
\wedge \psi(x))]$. The problem here is that too much is being demanded:
such a definition would be almost impossible to meet. Consider a very
simple property, such as ``x = $\emptyset$". What would be a $\Delta$
definition of this, according to the above notion of $\Delta$? If $\phi(x)$
is ``x = $\emptyset$" and $\psi(x)$ is ``x $\not = \emptyset$" then we
certainly do not have $\forall x \; (\phi(x) \vee \psi(x))$. To address
these (and other) problems, we will be content with defining $\Delta_{1}$,
noting that $\Sigma$ resp. $\Pi$ formulas are under IKP provably equivalent
to $\Sigma_{1}$ resp. $\Pi_{1}$ formulas. A $\Delta_{1}$ property of a
variable x is given by a pair of $\Delta_{0}$ formulas $\phi(x,y),
\psi(x,y)$ with free variables x,y such that
$$\forall x [\exists y \neg \neg (\phi(x,y) \vee \psi(x,y)) \; \wedge
\forall y \neg (\phi(x,y) \wedge \psi(x,y))].
$$
With this notion of $\Delta_{1}$, $\Delta_{1}$ Comprehension is provable in
IKP; see \cite{barwise}\ theorem I.4.5.

That a function
defined by $\Sigma$ recursion is $\Sigma_{1}$ definable holds in both KP and
IKP; again, the proof from \cite{barwise}, theorem I.6.4 holds.
The Second Recursion Theorem, as stated and proved in \cite{barwise} V.2, pretty
much stands as is, except that when syntax and semantics are introduced in ch.
III all of the standard Boolean connectives must be included as primitives
(since, for instance, in our context $\rightarrow$ cannot be defined in
terms of $\neg$ and $\vee$). The fact that the least fixed point of a
positive inductive $\Sigma$ definition is
$\Sigma_{1}$ definable actually requires a bit of care in one point, so the
development of this theory will be summarized here, with the reader referred to
\cite{barwise} for the details.

In IZF, suppose $\Gamma : {\cal P}(X) \rightarrow {\cal P}(X)$ is a monotone
inductive operator on X. Consider $\{ Y | \Gamma(Y) \subseteq Y \}$. This
set is
non-empty,
since it contains X itself as a member. Let $\Gamma_{fix} = \bigcap \{ Y |
\Gamma(Y) \subseteq Y \}$.
By the monotonicity of $\Gamma, \Gamma(\Gamma_{fix}) \subseteq \Gamma_{fix}$.
Applying
$\Gamma$ again we get $\Gamma \Gamma (\Gamma_{fix}) \subseteq \Gamma
\Gamma_{fix}$, so
$\Gamma (\Gamma_{fix}) \in \{ Y | \Gamma(Y) \subseteq Y \}$, and $\Gamma_{fix}
\subseteq \Gamma (\Gamma_{fix})$. Hence $\Gamma_{fix} = \Gamma (\Gamma_{fix})$,
and
$\Gamma_{fix}$ is a fixed point. Moreover, by its definition, $\Gamma_{fix}$ is
the least
fixed point. $\Gamma_{fix}$ can also be defined from the bottom up. Inductively
on
ordinals $\beta$, let $\Gamma_{< \beta} = \bigcup_{\alpha < \beta}
\Gamma_{\alpha},
\Gamma_{\beta} = \Gamma (\Gamma_{< \beta})$. (Note the standard way to
induct on ordinals
intuitionistically, which avoids the successor-or-limit case split.) Let
$\Gamma_{\infty}
= \bigcup_{\beta \in ORD} \Gamma_{\beta}$, which can be shown to be a set using
Bounding.
(Note that Replacement will not do! Let $\psi (Y, \beta)$ be ``Y =
$\Gamma_{\beta}$";
since such a $\beta$ cannot be uniquely chosen, we need (Comprehension and)
Bounding
in order to get a range for $\psi$, a set A of ordinals such that $\forall Y
\subseteq X$ if $\exists
\beta$ Y = $\Gamma_{\beta}$ then $\exists \beta \in$ A Y = $\Gamma_{\beta}$.
Letting $\gamma$
= TC(A), $\Gamma_{\infty} = \Gamma_{\gamma}$. This is where the current
argument
differs
from its classical version, and recurs when discussing admissible sets proper.)
It's easy to see that $\Gamma_{\beta} \subseteq \Gamma_{fix}$, inductively on
$\beta$, and that $\Gamma_{\infty}$ is a fixed point, making $\Gamma_{\infty}$
equal
to $\Gamma_{fix}$.

Now particularize to the case where $\Gamma$ is given by an X-positive $\Sigma$
formula $\psi$:
$\Gamma$(Y) = \{ x $\in$ X $\mid \psi$(x, X) \}. (What follows is adapted from
\cite{barwise}, VI.2.6.) If {\bf M} is admissible (i.e. {\bf M}
$\models$ IKP), then, letting $\alpha$ = ORD({\bf M}), $\Gamma_{\infty} =
\Gamma_{< \alpha}$,
as follows. In {\bf M}, the relation ``x $\in \Gamma_{\beta}$" is a $\Sigma_{1}$
relation, using the Second Recursion Theorem. So $\phi(x, \Gamma_{<\alpha})$ is
a $\Sigma$
relation, where ``y $\in \Gamma_{<\alpha}$" is interpreted as ``$\exists \beta$
y $\in
\Gamma_{\beta}$". If $\phi(x, \Gamma_{<\alpha})$ holds, then, by $\Sigma$
Bounding,
there is a set of ordinals A $\in$ {\bf M} such that $\phi(x,
\bigcup_{\beta \in
A} \Gamma_{\beta})$
holds. Letting $\gamma$ = TC(A), $\phi(x, \Gamma_{<\gamma})$, and x $\in
\Gamma_{\gamma}$.
So $\Gamma_{\alpha} = \Gamma_{<\alpha}$ is a fixed point. Since
$\Gamma_{<\alpha}
\subseteq \Gamma_{\infty}$ too, $\Gamma_{<\alpha} = \Gamma_{\infty}$. It has
already been
observed that $\Gamma_{<\alpha}$ is $\Sigma_{1}$ definable over {\bf M}.

So where are the differences between KP and IKP? Typically properties that a
classical
set theorist identifies automatically become inequivalent in an intuitionistic
setting. No one
would believe that every property around admissibility has the same status in
IKP.

If one of the pleasures of intuitionism is to surprise our intuitions, then
you're in for a treat.
You would have expected to find some difference between KP and IKP among the
most
common properties, those already cited. So there's no use having the
differences
there,
where you're already looking for them. Rather, the differences show up where
you'd never
think to check, right under your nose. Consider the basic axiom, $\Delta_{0}$
Bounding:
$$
\forall x \in A \; \exists y \; \phi(x, y) \rightarrow \exists B \; \forall x
\in A \; \exists y \in B \;
\phi(x, y), \; \phi \; \in \; \Delta_{0}.
$$
Consider the contrapositive with the negations pushed through (the ``classical
contrapositive'', with $\phi$ absorbing the negation):
$$
\forall B \; \exists x \in A \; \forall y \in B \; \phi(x, y) \rightarrow 
\exists x \in A \; \forall y \;
\phi(x, y), \; \phi \; \in \; \Delta_{0}.
$$
We call this latter property $\Delta_{0}$ Uniformity. Of course classically
$\Delta_{0}$
Uniformity and $\Delta_{0}$ Bounding are equivalent as contrapositives, but
intuitionistically they're not, as we'll see later.

An alternative to $\Delta_{0}$ Bounding classically is $\Sigma$ Reflection:
$$
\phi \rightarrow \exists A \; \phi^{(A)}, \; \phi \in \Sigma,
$$
where the superscript means bound all as yet unbound quantifiers by the
superscript. The classical contrapositive is:
$$
\forall A \; \phi^{(A)} \rightarrow \phi, \; \phi \in \Pi.
$$
The latter property is called $\Pi$ Persistence. Finally, regarding the last
alternative
axiomatization of KP, $\Sigma$ Bounding, its classical contrapositive is
$\Pi$ Uniformity.

\begin{proposition}
Over IKP - $\Delta_{0}$ Bounding, the following are equivalent:
\begin{enumerate}
\item $\Pi$ Persistence
\item $\Pi$ Uniformity
\item $\Delta_{0}$ Uniformity.
\end{enumerate}
\end{proposition}

\begin{proof}
In what follows, we will rely heavily on the fact that, for $\phi$ a $\Pi$
formula,
if B $\subseteq$ A then $\phi^{(A)} \rightarrow \phi^{(B)}$. This is proved by
induction on $\Pi$ formulas. The same also holds for A = V: if B is a set, then
$\phi \rightarrow \phi^{(B)}$.

(1) $\rightarrow$ (2) Suppose that $\forall B \; \exists x \in A \; \forall y
\in B \;
\phi(x, y), \phi$ a $\Pi$ formula. By the above mentioned fact, $\forall B \;
\exists x \in A \; \forall y \in B \;
\phi(x, y)^{(B)}$, i.e. $\forall B \; (\exists x \in A \; \forall y \; \phi(x,
y))^{(B)}$.
By $\Pi$ Persistence, $\exists x \in A \; \forall y \; \phi(x, y)$, as was
to be
shown.

(2) $\rightarrow$ (3) trivial

(3) $\rightarrow$ (1) We show that $\Pi$ Persistence holds for all $\Pi$
formulas $\phi$
by induction on $\phi$.

The base case, $\phi \; \Delta_{0}$, is trivial, since $\phi^{(A)} = \phi$.

Suppose $\phi = \psi_{0} \vee \psi_{1}$. Assume $\forall A (\psi_{0} \vee
\psi_{1})^{(A)}$,
which equals $\forall A (\psi_{0}^{(A)} \vee \psi_{1}^{(A)})$. We claim that
$\forall B \; \exists x \in \{0, 1\} \; \forall y \in B \; ((x=0 \wedge
\psi_{0}^{(y)})
\vee (x=1 \wedge \psi_{1}^{(y)}))$. To see this, by assumption $\forall A
(\psi_{0}^{(A)} \vee \psi_{1}^{(A)})$;
letting be A be $\bigcup$ B, we have $\psi_{0}^{(\bigcup B)} \vee
\psi_{1}^{(\bigcup B)}$.
If $\psi_{0}^{(\bigcup B)}$ let x be 0, and note that $\psi_{0}^{(\bigcup B)}
\rightarrow \psi_{0}^{(y)}$ since y $\subseteq \bigcup$ B. Similarly if
$\psi_{1}^{(\bigcup B)}$.
Applying $\Delta_{0}$ Uniformity to the claim, we get
$\exists x \in \{0, 1\} \; \forall y \; ((x=0 \wedge \psi_{0}^{(y)})
\vee (x=1 \wedge \psi_{1}^{(y)}))$. If the value of x which witnesses this
sentence
is 0, then $\forall y \; \psi_{0}^{(y)}$, and by induction $\psi_{0}$.
Similarly, if
the value of x which witnesses this sentence is 1, we get $\psi_{1}$. Since
either
0 or 1 witnesses this sentence, $\psi_{0} \vee \psi_{1}$, as was to be shown.

Suppose $\phi = \exists u \in v \; \psi$. Assume $\forall A (\exists u \in v \;
\psi)^{(A)}$,
which equals $\forall A \; \exists u \in v \; \psi^{(A)}$. We claim $\forall B
\;
\exists u \in v \;
\forall y \in B \; \psi^{(y)}$. To see this, again use the assumption with A =
$\bigcup$B. By $\Delta_{0}$ Uniformity, $\exists u \in v \; \forall y \;
\psi^{(y)}$,
and by induction $\exists u \in v \; \psi$.

The other cases are easier. If $\phi = \psi_{0} \wedge \psi_{1}$,
$\forall A (\psi_{0} \wedge \psi_{1})^{(A)} = \forall A (\psi_{0}^{(A)} \wedge
\psi_{1}^{(A)})
\rightarrow \forall A \psi_{0}^{(A)} \wedge \forall A \psi_{1}^{(A)}
\rightarrow$ (by induction) $\psi_{0} \wedge \psi_{1}$. If $\phi = \forall
u \in
v \; \psi,
\; \forall A \; (\forall u \in v \; \psi)^{(A)} = \forall A \; \forall u \in v
\; \psi^{(A)}
\rightarrow \forall u \in v \; \forall A \; \psi^{(A)} \rightarrow$ (by
induction)
$\forall u \in v \; \psi$. Finally, if $\phi = \forall u \, \psi$, suppose
$\forall A \, (\forall u \, \psi)^{(A)}$\ which equals $\forall A \,
(\forall u
\in A \, \psi^{(A)})$.
It follows that $\forall u \, \forall B \, \psi^{(B)}$: for arbitrary values v
and C
for u and B respectively, let A = C $\cup$ \{v\} in the assumption above,
yielding
$\forall u \in C \cup \{v\} \, \psi^{(C \cup \{v\})}(u)$. In particular,
for u =
v,
$\psi^{(C \cup \{v\})}(v)$, and, shrinking the bound, $\psi^{(C)}(v)$. So
$\forall u \, \forall B \, \psi^{(B)}(u)$, and, by induction, $\forall u \,
\psi(u)$.
\end{proof}

In what follows, $\Pi$ Persistence will refer to the theory IKP - $\Delta_{0}$
Bounding + $\Pi$ Persistence (or $\Pi$ Uniformity or $\Delta_{0}$ Uniformity,
from the proposition), as well as to the axiom scheme, except when such an
ambiguity
might cause confusion.

Although KP and $\Pi$ Persistence are equivalent classically, intuitionistically
they are (implicationally) incomparable, as follows. To see that IKP does
not imply $\Pi$ Persistence, consider the partial order
that has a bottom element
$\bot$ and $\omega$-many incomparable nodes $n$ ($n \geq$ 1) larger than
it. Let
the Kripke structure {\bf M} have L$_{\omega_{1}^{CK}}$ at $\bot$,
L$_{\omega_{n}^{CK}}$ ($\omega_{n}^{CK}$ being the nth admissible ordinal)
at $n$, and Id for transition
functions. $\bot \models$ IKP, as follows. {\bf M}$_{\bot}$ already
contains $\emptyset$ and $\omega$, and the universe at each node is closed
under pairing and union. These universes are also transitive sets, and the
$\in$-relation of {\bf M} is merely a restriction of $\in$ (of {\bf V}), so
Extensionality and Foundation hold. Both $\Delta_{0}$ Comprehension and
$\Delta_{0}$ Bounding are based on the fact that sets in {\bf M} don't
grow, or, to put it more formally, {\bf M}$_{n}$ is an end-extension of
{\bf M}$_{\bot}$. As a consequence, Excluded Middle holds in {\bf M} for
$\Delta_{0}$ formulas. So for $\Delta_{0}$ Comprehension, suppose $\phi(x)$
is a $\Delta_{0}$ formula with parameters from {\bf M}$_{\bot}$ =
L$_{\omega_{1}^{CK}}$, and X $\in$ {\bf M}$_{\bot}$. Let A = \{ x $\in$ X
$\mid \phi(x)$\}. A $\in$ L$_{\omega_{1}^{CK}}$ = {\bf M}$_{\bot}$, and
$\bot \models$ `` A = \{ x $\in$ X $\mid \phi(x)$\}''. Regarding
$\Delta_{0}$ Bounding, if $\bot \models$ ``f is total'' then $1 \models$
``f is total'',
and L$_{\omega_{1}^{CK}}$ $\models$ ``f is total''. By the admissibility of
L$_{\omega_{1}^{CK}}$, A = rng(f) $\in$ L$_{\omega_{1}^{CK}}$. By Excluded
Middle and Absoluteness for $\Delta_{0}$ formulas, $\bot \models$ ``A =
rng(f)''. But $\Delta_{0}$ Uniformity fails: $\bot \models \forall a \;
(\exists n \in \omega \; \forall x \in a$ if x is a sequence of admissible
ordinals then length(x) $<$ n), but $ \bot \not\models \exists n \in \omega \;
\forall x$ (if x is a sequence of admissible ordinals then length(x) $<$ n).

In the other direction, consider
$\omega$ as a partial order, and let $\alpha_{n}$ be an $\omega$-sequence
cofinal in
$\omega_{1}^{CK}$. Let the Kripke model {\bf M} have L$_{\alpha_{n}}$ at node n
(with Id
as the transition functions). It is easy to see that no node forces IKP: given
n, just pick a
witness that L$_{\alpha_{n}}$ is not admissible, such a witness being a
$\Delta_{0}$
function unbounded in L$_{\alpha_{n}}$. But $\Pi$ Persistence holds:
Suppose for a node $n$, $n \models ``\forall A \; \phi^{(A)}$'' (and hence
the same for all nodes $m, m \geq n$). For $A \in$ L$_{\alpha_{n}}$,
$\phi^{(A)}$ is $\Delta_{0}$, and {\bf M}$_{m}$ is an end-extension of {\bf
M}$_{n}$ for $m \geq n$; this means that the truth of $\phi^{(A)}$ can be
determined locally, i.e. L$_{\alpha_{n}} \models \phi^{(A)}$. ``A'' ranges
over all sets in L$_{\alpha_{m}}$ for all $m \geq n$, hence over
L$_{\omega_{1}^{CK}}$, so we have L$_{\omega_{1}^{CK}}$ $\models \forall A
\; \phi^{(A)}$. By $\Pi$ Persistence classically,
L$_{\omega_{1}^{CK}} \models \phi$. In showing that $n \models \phi$, when
unraveling
$\phi$ choices will have to be made, with $\exists$ and $\vee$. Use the
truth of
$\phi$ in
L$_{\omega_{1}^{CK}}$ as a guide.

So we have two different theories, IKP and $\Pi$ Persistence. Is there a
difference between the mathematics you can do in them? Yes,
again by duality. We have already seen that in IKP the least fixed point of a
positive
inductive $\Sigma$ definition is $\Sigma_{1}$ definable. Classically it would
follow that the
greatest fixed point of a positive inductive $\Pi$ definition is $\Pi_{1}$
definable. However,
intuitionistically, this latter property seems to require $\Pi$ Persistence as
well
as IKP.

In a bit more detail, if $\Gamma \, : \, {\cal P}(X) \rightarrow {\cal
P}(X)$ is
a monotone
inductive operator on X, let A be $\{ Y | Y \subseteq \Gamma(Y) \}$ and
$\Gamma_{fix}
= \bigcup A$. By the monotonicity of $\Gamma, \Gamma_{fix} \subseteq
\Gamma(\Gamma_{fix})$.
Applying $\Gamma$ again, $\Gamma(\Gamma_{fix}) \subseteq
\Gamma(\Gamma(\Gamma_{fix}))$,
so $\Gamma(\Gamma_{fix}) \in$ A and $\Gamma(\Gamma_{fix}) \subseteq
\Gamma_{fix}$.
Hence $\Gamma_{fix} = \Gamma(\Gamma_{fix})$, and $\Gamma_{fix}$ is a fixed
point.
Moreover, by its definition, $\Gamma_{fix}$ is the greatest fixed point. From
the
bottom up, let $\Gamma_{<\beta} = \bigcap_{\alpha < \beta} \Gamma_{\alpha}$,
$\Gamma_{\beta} = \Gamma(\Gamma_{<\beta})$, and $\Gamma_{\infty} =
\bigcap_{\beta \in ORD} \Gamma_{\beta}$.
$\Gamma_{fix} \subseteq \Gamma_{\beta}$, by induction on $\beta$, and
$\Gamma_{\infty}$
is a fixed point, so $\Gamma_{fix} = \Gamma_{\infty}$. If $\Gamma$ is given
by a
$\Pi$ formula $\phi$, then, over IKP + $\Pi$ Persistence, ``x $\in
\Gamma_{\beta}$"
is a $\Pi$ relation, by the duals of the arguments for the $\Sigma$ case.
(This can be seen by adapting \cite{barwise}, III.1 and V.1 and 2. The
satisfaction
relation {\bf M} $\models \phi$[s] is $\Delta_{1}$ over IKP. That the universal
$\Pi$ predicate $\Pi$-Sat($\phi$, s) can be defined as ``$\forall A
\phi^{(A)}$[s]"
uses $\Pi$ Persistence. This fact is then plugged into the dual to the proof of
the Second Recursion Theorem to get that for all R-positive $\Pi$ formulas
$\phi$(x, y, R) there is a $\Pi$ formula $\psi$(x, y) such that $\psi$(x, y) iff
$\psi$(x, y, $\lambda$x.$\psi$(x, y)). This suffices to get ``x $\in
\Gamma_{\beta}$"
to be a $\Pi$ relation of x and $\beta$.) Let {\bf M} $\models$ IKP + $\Pi$
Persistence, and $\alpha$ = ORD({\bf M}). The claim is that $\Gamma_{<\alpha} =
\Gamma_{\alpha}$. To see this, suppose x $\in \Gamma_{<\alpha}$, i.e. {\bf M}
$\models \forall \beta \phi (x, \Gamma_{<\beta})$, a $\Pi$ relation. Letting
$\Gamma_{A}$ be $\bigcap_{\alpha \in A \cap ORD} \Gamma_{\alpha}$, {\bf M}
$\models \forall A \phi(x, \Gamma_{A})$ (because $\beta$ can be chosen to be
TC(A) $\cap$ ORD, and then $\Gamma_{<\beta} \subseteq \Gamma_{A}$), and {\bf M}
$\models \forall A \; \phi^{(A)}(x, \Gamma_{A})$. By $\Pi$ Persistence,
{\bf M} $\models \phi(x, \Gamma_{<\alpha})$, so x $\in \Gamma_{\alpha}$ and
$\Gamma_{<\alpha} \subseteq \Gamma_{\alpha}$. On general principles
$\Gamma_{\alpha}
\subseteq \Gamma_{<\alpha}$, so $\Gamma_{\alpha} = \Gamma_{<\alpha}$, and
$\Gamma_{<\alpha}$ is a fixed point. Since $\Gamma_{<\alpha} \supseteq
\Gamma_{\infty}$,
$\Gamma_{<\alpha} = \Gamma_{\infty}$.

So, while IKP suffices to get least fixed points to be $\Sigma_{1}$ definable,
it seems as though $\Pi$ Persistence is necessary to get greatest fixed points
to be $\Pi_{1}$ definable (although, to be fair, such necessity has yet to be
proven). Furthermore, the second of
the models above, intended to show that $\Pi$ Persistence does not imply IKP,
also shows
that $\Pi$ Persistence doesn't prove that lfp's are $\Sigma_{1}$ definable
(Kleene's {\it O}
has an appropriate inductive definition but is not $\Sigma_{1}$ definable over
{\bf M}.). So
these two constructions, least and greatest fixed point, so near to each other
classically,
are apparently more easily splittable intuitionistically. This matter will be
pursued in the
questions at the end of this paper.

It bears observation that $\Pi$ Persistence, as a theory, is quite weak, in that
one cannot easily construct sets in it. For instance, it does not even prove the
totality of the function $\alpha \mapsto L_{\alpha}$, as the following model
shows. Let $\langle \, \alpha_{n} \mid n \in \omega \, \rangle$ be a strictly
increasing sequence
of limit ordinals cofinal in $\omega_{1}^{CK}$. For X a set, let Pair(X) be the
set of all pairs from X (including singletons, as degenerate pairs), Union(X)
the set of all unions from X (i.e. \{$\bigcup x \mid x \in$ X\}), and
$\Delta_{0}$(X)
the set of all $\Delta_{0}$-definable subsets of members of X (assume here that
X
is transitive, just for simplicity). Let Close(X) be X $\cup$ Pair(X) $\cup$
Union(X) $\cup$ $\Delta_{0}$(X). The Kripke model under construction has for its
partial order an increasing $\omega$-sequence. For node n, start with X$_{0}$ =
$L_{\alpha_{n}} \cup \alpha_{n+1}$, and let X$_{m+1}$ = Close(X$_{m}$). The
n$^{th}$ node is $\bigcup_{m}$X$_{m}$ (with the inclusion function as the
transition functions). IKP - $\Delta_{0}$ Bounding is easily seen to hold, as is
$\Pi$ Persistence, by the same argument as in the second model above. But
$\alpha_{n}$ is an ordinal at every node n and L$_{\alpha_{n}}$ isn't a set
until node n+1. (To see that L$_{\alpha_{n}}$ is not a set at node n, show
inductively on m that for all finite Y $\subseteq$ X$_{m}$
L$_{\alpha_{n}} \not\subseteq$ TC(Y).)

It is an interesting question whether IKP + $\Pi$ Persistence -
$\Delta_{0}$ Bounding proves the $\Pi_{1}$ definability of greatest fixed
points of positive inductive $\Pi$ definitions, or whether the full power
of IKP + $\Pi$ Persistence is needed. The status of $\Pi$ Persistence will
be further
pursued in the questions at the end.

\section{Classical Friends}
In the following section we will turn to the intuitionistic versions of various
theories related
to KP. Since these classical theories are not all well known, though, a brief
discussion of
them for their own sakes is in order.

The axioms of interest to us are:
\begin{enumerate}
\item	$\Sigma_{1}$ Dependent Choice: If $\phi(x, y)$ is $\Sigma_{1}$, and
$\forall x \;
\exists y \; \phi(x, y)$, then there is a function f with domain $\omega$ such
that $\forall n \;
\phi$(f(n), f(n+1)).
\item	Resolvability: There is a $\Delta_{1}$ definable function f on the
ordinals such
that V = $\bigcup$range(f).
\item	$\Pi_{2}$ Reflection: If $\forall x \; \exists y \; \phi(x, y), \;
\phi \; \in \;
\Delta_{0}$, then
for every A there is a B $\supseteq$ A such that $\forall x \in B \; \exists y
\in B \; \phi(x, y)$.
\item	$\Delta_{0}$ Bounding: from KP.
\end{enumerate}

The implicational relations among these theories are as follows:
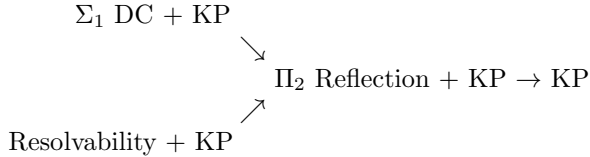
\begin{tabbing}
$\;\;\;\;\;\;\;\;\;\Sigma_{1}$ DC + KP \= \\
      					\> $\searrow$ \= \\
           \>            \>$\Pi_{2}$  Reflection + KP  $\rightarrow$  KP \\
           \> $\nearrow$ \\
Resolvability + KP
\end{tabbing}
Over the base theory KP - $\Delta_{0}$ Bounding,
Resolvability alone implies nothing of interest; after all, L$_{\lambda}$ for
any arbitrary limit
$\lambda$ is resolvable. Similarly, $\Sigma_{1}$ DC alone is rather weak, as
L$_{\omega_{1} \cdot 2}$ models $\Sigma_{1}$ DC (because cf($\omega_{1}
\cdot 2) > \omega$). However, over this base theory
$\Pi_{2}$ Reflection itself picks up $\Delta_{0}$ Bounding.

All of these implications are easy to prove; for details see \cite{barwise}.
Furthermore, in L
they are all equivalent; this is also easy, and can be found in \cite{barwise}.
However, if
an implication does not follow from the preceding then it is not in general
true. Some of
these non-implications are easy to see, and must certainly already be known if
not written
down; others are actually a bit tricky. In any event these proofs are not so
easy to find in
the literature, and they also provide the models for the proofs in the next
section, so it is
well to include them here. (For background on forcing, as well as arguments similar 
to those that follow, including the use of symmetric submodels to falsify the Axiom of
Choice, the reader is referred to \cite{JechMF} and \cite{JechST}.)

1. Resolvability + KP + $\neg\Sigma_{1}$ DC: Adjoin to L$_{\omega_{1}^{CK}} \;$
(or, for that matter, to L $\models$ ZF, for the stronger result Resolvability + ZF +
$\neg\Sigma_{1}$ DC) $\omega$-many Cohen reals (that is, reals generic for the forcing
partial order 2$^{< \omega}$), as well as the set (not 
sequence!) G consisting of all those reals. This is a
symmetric submodel of the model obtained by the set forcing to add an
$\omega$-sequence of
Cohen reals, hence is admissible. It is already a standard argument (see \cite{JechST})
that not even all of ZF is
enough to build
an $\omega$-sequence of distinct elements of G. This violation of the Axiom of Choice
is actually a violation of $\Sigma_{1}$ DC: let $\phi(x,y)$ be ``if $x$ is a sequence of
distinct elements of G then $y$ is one also and is a proper end-extension of $x$".
The function which shows Resolvability is f($\alpha$) =
L$_{\alpha}$[G].

2. $\Sigma_{1}$ DC + KP + $\neg$Resolvability: Adjoin to L$_{\omega_{1}^{CK}}$
any finite
subset G$_{fin}$ of a countable set G of Cohen reals, and take the union of all such
adjunctions: {\bf M} = $\bigcup_{G_{fin} \subseteq G, \; G_{fin} finite}$ 
L$_{\omega_{1}^{CK}}$[G$_{fin}$].
$\Sigma_{1}$ DC and $\Delta_{0}$ Bounding hold, as follows. It suffices to consider
formulas of the form $\exists y \phi(x,y)$, with $\phi(x,y)$ a $\Delta_{0}$ formula.
If $\exists y \phi(x,y)$ holds for a particular set $x$, then let G$_{fin} \subseteq$ G
be the finite set of those reals used in the construction of $x$ and of $\phi$'s
parameters (i.e. $\exists y \phi(x,y)$ is a formula over L$_{\omega_{1}^{CK}}$[G$_{fin}$]).
We claim there is a witness $y$ not merely in {\bf M}, as hypothesized,
but already in L$_{\omega_{1}^{CK}}$[G$_{fin}$]. To see this, let $y$ witness
$\exists y \phi(x,y)$ in {\bf M}. Let G$_{y} \supseteq$ G$_{fin}$ be a finite set 
such that $y \in$ L$_{\omega_{1}^{CK}}$[G$_{y}$]. By the absoluteness of $\Delta_{0}$
formulas, L$_{\omega_{1}^{CK}}$[G$_{y}$] $\models \; \phi(x,y)$; again by absoluteness,
L$_{\alpha}$[G$_{y}$] $\models \; \phi(x,y)$ for some $\alpha < \omega_{1}^{CK}$.
G$_{y} \setminus$ G$_{fin}$ is generic over L$_{\omega_{1}^{CK}}$[G$_{fin}$] (for the 
forcing p.o. which is the product of finitely many copies of 2$^{< \omega}$), hence the last 
assertion is forced over L$_{\omega_{1}^{CK}}$[G$_{fin}$] by some condition p: 
p $\mid \vdash$ ``L$_{\alpha}$[G$_{y}$] 
$\models \; \phi(x,\mathring y)$", where $\mathring y$ is a name for $y$. 
In L$_{\omega_{1}^{CK}}$[G$_{fin}$] a generic over
L$_{\alpha}$[G$_{fin}$] through p can be built (by the countability of L$_{\alpha}$[G$_{fin}$] 
in L$_{\omega_{1}^{CK}}$[G$_{fin}$]); let G$_{\alpha}$ be such a generic. Then
L$_{\alpha}$[G$_{fin}$, G$_{\alpha}$] $\models \; \phi(x,\mathring y$(G$_{\alpha}))$, 
where ``$\mathring y$(G$_{\alpha}$)" refers to the interpretation of $\mathring y$ with the 
canonical name for the generic interpreted as G$_{\alpha}$. By the absoluteness of 
$\Delta_{0}$ formulas L$_{\omega_{1}^{CK}}$[G$_{fin}$] 
$\models \; \phi(x,\mathring y$(G$_{\alpha}$)), and L$_{\omega_{1}^{CK}}$[G$_{fin}$] 
$\models \; \exists y \; \phi(x,y)$. Now it's easy to see that $\Sigma_{1}$ DC holds in 
{\bf M}. If {\bf M} $\models \; ``\forall x \exists y \phi(x,y)", \phi \in \Delta_{0}$, then
L$_{\omega_{1}^{CK}}$[G$_{fin}$] models the same. L$_{\omega_{1}^{CK}}$[G$_{fin}$] also satisfies
$\Sigma_{1}$ DC, so contains an appropriate $\omega$-sequence, which by absoluteness also
works in {\bf M}. $\Delta_{0}$ Bounding is similar.

Resolvability fails, as follows. Let f be a function over {\bf M} (given via a $\Delta_{1}$
definition)
with domain ORD$^{\bf M}$. Let G$_{fin}$ be a finite subset of G such that all of f's 
parameters are in L$_{\omega_{1}^{CK}}$[G$_{fin}$]. By the same argument as in the last 
paragraph, since f is total over {\bf M}, the same $\Delta_{1}$ definition produces the same
total function over L$_{\omega_{1}^{CK}}$[G$_{fin}$]. So $\bigcup$ range(f) $\subseteq$
L$_{\omega_{1}^{CK}}$[G$_{fin}$], which itself is a proper subset of {\bf M}, so
$\bigcup$ range(f) $\not =$ {\bf M}.

3. $\Pi_{2}$-Reflection + $\neg \Sigma_{1}$DC + $\neg$Resolvability: This is a
combination of the previous two arguments. To construct {\bf M}, first adjoin 
an infinite set G of
Cohen reals, then all possible finite subsets of another such infinite set H.
$\Sigma_{1}$ DC and Resolvability fail as before. $\Pi_{2}$ Reflection holds for 
much the same reason that
$\Sigma_{1}$ DC did in the previous example. If {\bf M} $\models \exists y \phi(x, y)$ then all
you need to
build such a y is G and the finite subset H$_{fin}$ of H used for x and $\phi$. 
So if {\bf M} $\models ``\forall x \exists y \; \phi(x,y)", \; \phi \in \Delta_{0}$,
and if A $\in$ {\bf M},
then L$_{\omega_{1}^{CK}}$[G, H$_{fin}$] $\models \; ``\forall x \exists y \; \phi(x,y)"$
(H$_{fin}$ merely being large enough to contain all of the reals in H needed to
construct $\phi$'s parameters and A). Since L$_{\omega_{1}^{CK}}$[G, H$_{fin}$] satisfies
$\Pi_{2}$ Reflection, it contains an appropriate B $\supseteq$ A, as desired.

4. KP + $\neg \Pi_{2}$-Reflection: The model from \cite{metoo} suffices, and
will be
described briefly here. A Steel generic tree over L$_{\omega_{1}^{CK}}$ is not
well-founded but has no infinite descending sequence in the generic extension.
This
forcing can be modified slightly to include a distinguished path through the
tree which
does not destroy admissibility. Force an $\omega$-sequence of such trees
$<T_{n}>$ and
paths $<B_{n}>$. Let W be a recursive linear ordering of $\omega$ which is not
well-founded but has no hyperarithmetic infinite descending sequence. Think of
the trees
as ordered by W. Let {\bf M}$_{i}$ be L$_{\omega_{1}^{CK}}$[$<T_{n} \mid n \in
\omega>,
<B_{n} \mid n \geq_{W} i>$]. {\bf M} = $\bigcup_{i \in I}$ {\bf M}$_{i}$ for I
some final
segment of W with no least element, chosen exactly so that {\bf M} is
admissible. In
\cite{metoo} it was observed that {\bf M} falsifies $\Sigma_{1}$ DC; in fact
it's even
$\Pi_{2}$ Reflection that fails: {\bf M} $\models ``\forall j \;$  if there
is a
branch through
T$_{j}$ then $\exists i<_{W}j$ and a branch through T$_{i}$'', but any set in
{\bf M} is a set
in some {\bf M}$_{i}$, in which W is still well-founded.

\section{Intuitionistic Models}

As remarked in the introduction, even though the constructions just given show
that none
of these theories are equivalent to each other in general, in L they all are.
However,
intuitionistically the ordinals do funny things. With Kripke models there is an
extra degree
of freedom, the choice of the partial order. If the underlying partial order is
chosen to be
some forcing partial order, then by coding forcing conditions as small ordinals one
could
reasonably hope to get a generic G, although still not in V, actually in
L$^{V[G]}$. This is
basically the approach followed in the upcoming arguments. Notice, though, that
these
intuitions are not made precise. In particular, of the four constructions from
the preceding
section, only three are adapted here, because we couldn't see how to do the
fourth. The
omitted one is $\Sigma_{1}$ DC + KP + V=L + $\neg$Resolvability, the problem
being
$\Sigma_{1}$ DC. There's no trouble picking some least set in classical L, but
IL does not
apparently come equipped with a well-founded linear order.

\subsection{General Considerations} \label{general}

There will be some common constructions in these models. It would be well to
isolate
them, as well as to give some of the standard background on IORD and IL.

\subsubsection{Intuitionistic Ordinals}

Intuitionistically, an ordinal is a transitive set of transitive sets. The
standard weak counter-example to show that the ordinals are not linearly
ordered goes as follows. Let $\alpha = \{0 \mid \phi\}$. $\alpha$ is an
ordinal. If the
ordinals are linearly ordered, then $\alpha \in 0, \alpha = 0$, or $0 \in
\alpha$.
The first case is not possible, since 0 is the empty set. In the second case,
$\neg \phi$. In the last case, $\phi$. So $\phi \vee \neg \phi$. Hence if
Excluded Middle is to fail, then the ordinals are not linearly ordered. It is
not hard to give examples of
non-classical ordinals in Kripke models. Let {\it P} be any partial order. A
trivial Kripke structure on {\it P}
has some universe of sets M (with M's $\in$-relation) at each node and the identity 
function as the
transition
functions. (Think of M as a model of ZF, or a weaker set theory, or even the universe
V for a class-sized Kripke structure. A class-sized Kripke structure is like a Kripke
structure, but rather than requiring that the relations and functions in question (e.g.
``is a set at node $\sigma$", or the transition functions between nodes) be sets, they
must merely be definable relations on a definable subset of V.)
Consider a larger or ambient Kripke structure, that is, a structure with the same partial
order {\it P} but more sets at each node. Any set in this ambient structure which at each 
node contains only ordinals
from the trivial Kripke structure is an intuitionistic ordinal which cannot be
construed as
classical.

\subsubsection{Intuitionistic Definability}
\label{section-int-def}

Regarding IL, definability can be developed within IZF much as it is
classically, so def(X), the set of definable subsets of a set X, makes
perfect sense. As usual, the definition of
L$_{\alpha}$ cannot
depend on whether $\alpha$ is a limit or a successor, since intuitionistically
not every
ordinal is a limit or a successor. So L$_{\alpha}$ is $\bigcup_{\beta \in
\alpha}$
def(L$_{\beta}$). For details see \cite{me}. In what follows, however, we will
typically not be dealing with L as an inner model of an already given model of
IZF, which
is merely the semantic version of L's axiomatic development in IZF, hence
implicit
in \cite{me}. Rather, L will be handled {\it sui generis}: given a Kripke
structure {\bf M}
for the language of set theory, we will have occasion to consider def({\bf M}).
While not difficult, this could stand some exposition.

So suppose {\bf M} is an extensional ($x = y$ iff $\forall z \; z \in x
\leftrightarrow z \in y$) Kripke structure (with relations $\in$ and
=),
with underlying partial order {\it P}, universe {\bf M}$_{\sigma}$ at node
$\sigma$,
and transition functions f$_{\sigma \tau}$ : {\bf M}$_{\sigma} \rightarrow$
{\bf M}$_{\tau}$. In def({\bf M}), a set at node $\sigma$ is named by a formula
$\phi$(x) with free variable x and parameters from {\bf M}. A set is actually an
equivalence class of such formulas, in which $\phi \sim_{\sigma} \psi$ if
$\sigma \models_{\bf M} \forall x \; [\phi(x) \leftrightarrow \psi(x)]$.
f$_{\sigma \tau}$ : def({\bf M}$_{\sigma}$) $\rightarrow$
def({\bf M}$_{\tau}$) is given by the equation

$$f_{\sigma \tau}([\phi(x, y_{1}, ... , y_{n})]_{\sigma}) =
[\phi(x, f_{\sigma \tau}(y_{1}), ... , f_{\sigma \tau}(y_{n}))]_{\tau},
$$
where $y_{1}, ... , y_{n}$ are the parameters, and the notation f$_{\sigma
\tau}$
is used ambiguously for transition functions in both {\bf M} and def({\bf M}).
f$_{\sigma \tau}$ is well-defined, because if $\phi \sim_{\sigma} \psi$, that
is,
if $\sigma \models_{\bf M} \forall x \; [\phi(x, y_{1}, ... , y_{n})
\leftrightarrow \psi(x, z_{1}, ... , z_{m})]$, then, by general lemmas on Kripke
models, $\tau \models_{\bf M} \forall x \; [\phi(x, f_{\sigma \tau}(y_{1}),
... ,
f_{\sigma \tau}(y_{n}))
\leftrightarrow \psi(x, f_{\sigma \tau}(z_{1}), ... , f_{\sigma \tau}(z_{m}))]$
($\sigma \leq \tau$), or $\phi \sim_{\tau} \psi$. {\bf M}$_{\sigma}$ can be
embedded in def({\bf M}$_{\sigma}$), with y being identified with the formula
``x $\in$ y". This much being said, $\in$ can now be interpreted:
[$\phi$]$_{\sigma}
\in$ [$\psi$]$_{\sigma}$ iff [$\phi$]$_{\sigma}$ = [y]$_{\sigma}$ for some y
$\in$ {\bf M}$_{\sigma}$ and $\sigma \models_{\bf M} \psi(y)$. (Notice that this
definition is independent of the choice of representative for
[$\psi$]$_{\sigma}$,
by the definition of the equivalence relation; furthermore, there is at
most one
y such
that y $\sim_{\sigma} \phi$, by the extensionality of {\bf M}.) Extensionality
holds in def({\bf M}): \\
$\sigma \models_{def({\bf M})} \forall x \; (x \in [\phi]_{\sigma}
\leftrightarrow
x \in [\psi]_{\sigma})$ iff \\
$\forall \tau \geq \sigma \; \forall x \in {\bf M}_{\tau}
\tau \models_{def({\bf M})}  [x] \in [\phi]_{\tau} \leftrightarrow
[x] \in [\psi]_{\tau}$ iff \\
$\forall \tau \geq \sigma \; \forall x \in {\bf M}_{\tau}
\tau \models_{\bf M}  \phi(x) \leftrightarrow \psi(x)$ iff \\
$\sigma \models_{\bf M}  \forall x \; [\phi(x) \leftrightarrow \psi(x)]$ iff \\
$[\phi]_{\sigma} = [\psi]_{\sigma}$.

To summarize, if {\bf M} is an extensional Kripke structure for the
language of set theory,
so is def({\bf M}), and {\bf M} can be canonically embedded in def({\bf M}).
Furthermore, def({\bf M}) can be construed as an end-extension of {\bf M}, since
if $\sigma \models_{def({\bf M})} [\phi]_{\sigma} \in [y]_{\sigma}$ then
$[\phi]_{\sigma} = [z]_{\sigma}$ for some z $\in$ {\bf M}$_{\sigma}$ and
$\sigma \models_{\bf M} z \in y$. Also, this process can be iterated through the
transfinite: starting with {\bf M}$_{0}$ = {\bf M}, let {\bf M}$_{\alpha + 1}$
= def({\bf M}$_{\alpha}$) and {\bf M}$_{\lambda}$ =
$\stackrel{lim}{\rightarrow}$
\{{\bf M}$_{\alpha} \mid \alpha < \lambda$\} (or, more colloquially, $\bigcup
_{\alpha < \lambda} {\bf M}_{\alpha}$). {\bf M}$_{\lambda}$ is also an
extensional
Kripke model end-extending each {\bf M}$_{\alpha}$.

Indeed, this process can be iterated not only along a well-founded linear
order,
but also along a well-founded partial order. Suppose that X is a Kripke
structure (for the language of set theory) with underlying partial order {\it
P}, and that the partial order defined by

$$(\sigma, x) \geq (\tau, y) \leftrightarrow \sigma \leq_{\it P} \tau
\wedge \tau \models y \in f_{\sigma \tau}(x)
$$
$(\sigma, \tau \in {\it P}, x \in X_{\sigma}, y \in X_{\tau})$ is well-founded.
Then ``def" can be iterated along this p.o., as follows. Suppose
inductively that
for all $(\tau, y) \leq (\sigma, x)$ def({\bf M}$_{ (\tau, y)}$) is
well-defined as a Kripke structure on the partial order {\it P}$_{\geq
\tau}$. Then {\bf M}$_{(\sigma, x)}$ can be defined as follows:

$$({\bf M}_{(\sigma, x)})_{\tau} = \bigcup_{\{y \mid \tau \models y \in
f_{\sigma \tau}(x)\}} (def({\bf M}_{(\tau, y)}))_{\tau}.
$$
Regarding the transition functions, consider $[\phi(z, y_{1}, ... ,
y_{n})]_{\tau} \in ({\bf M}_{(\sigma, x)})_{\tau}$. Then

$$[\phi(z, y_{1}, ... , y_{n})]_{\tau} \in (def(({\bf M}_{(\tau, y)}))_{\tau}
$$
where $\tau \models y \in f_{\sigma \tau}(x)$. Then let

$$f_{\tau \rho}([\phi(z, y_{1}, ... , y_{n})]_{\tau}) :=
[\phi(z, f_{\tau \rho}(y_{1}), ... , f_{\tau \rho}(y_{n}))]_{\rho},
$$
since
$$[\phi(z)]_{\rho} \in (def({\bf M}_{(\rho, f_{\tau \rho}(y))}))_{\rho}
\subseteq ({\bf M}_{(\sigma, x)})_{\rho},
$$
the latter inclusion because $\rho \models f_{\tau \rho}(y) \in f_{\tau
\rho}(f_{\sigma \tau}(x)) = f_{\sigma \rho}(x)$.

With {\bf M}$_{(\sigma, x)}$ now defined, def({\bf M}$_{(\sigma, x)}$)
follows, using the ``def" operator developed earlier in this section. The
final structure {\bf M}$_{X}$ is then given by
$$({\bf M}_{X})_{\sigma} = \bigcup_{x \in X_{\sigma}} (def({\bf
M}_{(\sigma, x)}))_{\sigma}.
$$

At this point a few words about absoluteness are in order. For starters,
the external concept ``def" developed above is equivalent to the internal
concept ``def". In more detail, suppose that {\bf V} is some
classical meta-universe, and {\bf M}$\in${\bf V} is an extensional Kripke structure for the
language of set theory, and {\bf N} = def({\bf M}), again in {\bf V}.
Suppose that {\bf K} is a Kripke structure with the same underlying partial
order as {\bf M} and {\bf N} such that {\bf M}, {\bf N} $\in$ {\bf K} (so
{\bf M} and {\bf N} are sets in the sense of {\bf K}). Then {\bf K}
$\models$ {\bf N} = def({\bf M}), where of course ``def" in this case is
the traditional, internal concept of definability. We will not prove this
here; hopefully the external development of ``def" was transparent enough
to make this assertion clear. This allows us in the following to be
ambiguous as to which version of ``def" is intended.

Furthermore, the same holds for the iterations of ``def", both into the
transfinite as well as along a well-founded partial order. In the first
case, suppose $\alpha$ is an ordinal and {\bf K} a Kripke structure for the
language of set theory (both in our classical meta-universe {\bf V}, of
course). Suppose that there is an $\alpha_{\bf K} \in {\bf K}_{\sigma},
\sigma \in$ {\it P}, such that $(\alpha, \in_{\bf V})$ and $(\alpha_{\bf
K}, \in_{\bf K_{\sigma}})$ are isomorphic, and that, for all $\tau \geq
\sigma$, f$_{\sigma\tau}$ : $\alpha_{\bf K} \rightarrow$
f$_{\sigma\tau}(\alpha_{\bf K})$ is an isomorphism. Then $\sigma \models
``\alpha_{\bf K}$ is an ordinal", and $\alpha_{\bf K}$ can be identified
with $\alpha$. Notice that this is nothing other than the extension of the
already common practice of using the notation ``0" to denote the empty set
both in {\bf V} and in a Kripke model, ``1" as \{0\} again in {\bf V} or
in a Kripke model, etc., and even ``$\omega$" for the $\omega$ of a Kripke 
model, so long as it contains only the internalizations of the standard
natural numbers. In addition, for {\bf M}, {\bf N} $\in$ {\bf
K$_{\sigma}$}, if {\bf N} = {\bf M}$_{\alpha}$ (the $\alpha$-fold iteration
of def, as defined above, of course interpreted in {\bf V}), then $\sigma
\models$ {\bf N} = {\bf M}$_{\alpha_{\bf K}}$. (This applies, of course,
only when {\bf K} has enough set-theoretic power to express $\alpha$-fold 
iteration of definability, e.g. in models of IKP (see \cite{me}).) In such a 
case we will feel
free to use the notation $\alpha$ for both the external and internal
ordinal.

Regarding iterations along well-founded partial orders, let X be as above
(that is, a Kripke structure over {\it P} well-founded in the manner previously
described). Let {\bf N} = {\bf M}$_{X}$, of course in {\bf V}, and let {\bf K}
be a Kripke structure containing X, {\bf M}, and {\bf N} as members. If {\bf K}
satisfies a sufficiently large fragment of IZF then the X-fold iteration of def 
(notation: X-def) would be definable and provably total. The inductive definition of ``X-def" 
is none other
than the already well-known notion of the iteration of definability along an ordinal
(i.e. {\bf L}$_{\alpha}$),
the point being that $\alpha$'s ordinalhood is unnecessary for the intelligibility of
the notion: X-def({\bf M}) = $\bigcup_{Y \in X}$ def(Y-def({\bf M})). The result, again
stated without proof, is that {\bf K} $\models$ {\bf N} = X-def({\bf M}).

A case of particular interest is when {\bf K} $\models$ ``X is an ordinal". Under this
circumstance we will use the more common notation $\alpha$ instead of X. 
For {\bf M} take the structure {\bf 0} = $\emptyset^{\bf K}$ (equivalently, let 
{\bf 0} be the unique function from {\it P} to $\{ \emptyset \}$). Then there is
already standard notation for $\alpha$-def({\bf 0}), namely {\bf L}$_{\alpha}$.
This being the case, we will use the
notation ``{\bf L}$_{\alpha}$" for {\bf 0}$_{\alpha}$ for those $\alpha$'s as
above which are ordinals in some Kripke structure, even if such Kripke structure
has not yet been introduced in the exposition. Then ``{\bf
L}$_{\alpha}$" can be interpreted either externally or internally, as the
occasion demands. (To keep the reader from being distracted by concerns that
any given $\alpha$ will indeed turn out to be an ordinal in some ambient Kripke
structure introduced in some unspecified future section, please observe that
ordinality is a $\Delta_{0}$
property for transitive sets, the only kind we will be considering, and
hence can be determined locally. We will even go so far as to say that a
transitive Kripke structure $\alpha$ {\it is} an ordinal (or a Kripke ordinal)
when it satisfies
the intuitionistic definition of such, a transitive set of transitive sets,
even in the absence of an ambient Kripke universe, knowing that such an $\alpha$
will be an ordinal in any ambient Kripke structure.)

\subsubsection{General Constructions}

In the following sections we will be constructing several Kripke models.
Here we
bring together several steps common to them, in order not to have to do them
more than once. In every case we will begin by giving the underlying partial
order {\it P}; the steps given here are generic and applicable to all {\it
P}'s,
assuming only that {\it P} has a bottom element $\bot$. (After corollary \ref{subsets}, we will
impose more restrictions on {\it P}. Of course for the constructions themselves
in future sections, we will specify {\it P} explicitly.) Then we will define certain
Kripke ordinals with underlying partial order {\it P}. These ordinals will be based 
on the generic reals from the classical forcing constructions of the last section;
they will be a kind of internalization of these reals, and will be very low in the power
set hierarchy (being collections of subsets of 1 = \{0\}). It should be noted that
the constructions here and in the coming sections take place of necessity in
{\bf
V}: the only Kripke structures in sight are those being built, and so of no use
in their own definitions.

Given a node $\sigma$ in {\it P}, let $1_{\sigma}$ be the set (Kripke set, if you
will, as $1_{\sigma}$ will end up being a set in a Kripke structure) that looks
like 0 unless
you're past or incompatible with $\sigma$, where it's 1:

$$[\tau \models x \in 1_{\sigma}] \leftrightarrow [\sigma < \tau \vee \sigma
\perp \tau]
\wedge [\tau \models x = \emptyset],
$$
or, if you prefer,
$$(1_{\sigma})_{\tau} = \emptyset \leftrightarrow \tau \leq \sigma
$$
$$(1_{\sigma})_{\tau} = \{\emptyset\} \leftrightarrow [\sigma < \tau \vee
\sigma \perp \tau],
$$
with the only possible transition functions (cf. the remark after definition 
\ref{branches}). $1_{\sigma}$ is an ordinal:
$\bot \models$ ``if x $\in 1_{\sigma}$ then x =
$\emptyset$'';
since nothing is in $\emptyset$, $1_{\sigma}$ is transitive, and $\emptyset$ is
transitive. (Recall that the intuitionistic definition of an ordinal as a
transitive set of transitive sets is $\Delta_{0}$ for transitive sets, and
so can be determined locally, without an ambient Kripke structure.) So
L$_{1_{\sigma}}$ is well-defined. What is L$_{1_{\sigma}}$?
If $\tau \models x \in L_{1_{\sigma}}$, then, for some $\beta, \tau \models x \in def(L_{\beta})
\wedge \beta \in 1_{\sigma}$. If $\tau \leq \sigma$ then $\tau \not\models \beta
\in 1_{\sigma}$, so $\tau \not\models x \in L_{1_{\sigma}}$ for any x, and
L$_{1_{\sigma}}$ looks empty. Otherwise $\tau \models \beta \in 1_{\sigma}$ iff
$\tau \models \beta = \emptyset$. $\bot \models L_{0} = \emptyset$, and $\bot
\models def(L_{0}) = \{\emptyset\}$. So if $\tau \models x \in L_{1_{\sigma}}$
then $\tau \models x = \emptyset$. In short, L$_{1_{\sigma}}$ = $1_{\sigma}$.

Furthermore, what's definable over 1$_{\sigma}$? Recall that when taking
definitions over a set, truth is evaluated in that set: a definable subset
of 1$_{\sigma}$ is one of the form \{ x $\in 1_{\sigma} \mid 1_{\sigma}
\models \phi$(x) \}, where $\phi$'s parameters must also come from
1$_{\sigma}$. So if both $\tau_{0}, \tau_{1}
\leq \sigma$ then $\tau_{0} \models ``1_{\sigma} \models \phi(x)"$ iff
$\tau_{1} \models ``1_{\sigma} \models \phi(x)"$; i.e. $\tau_{0}$ and $\tau_{1}$
force the same atomic facts about definable subsets of 1$_{\sigma}$. Similarly
if both $\tau_{0}$ and $\tau_{1}$ extend or are incompatible with $\sigma$.
Since 1$_{\sigma}$'s only possible member is $\emptyset$, the only possible
member of a
subset of 1$_{\sigma}$ is $\emptyset$; by the preceding remarks, either
$\emptyset$ is in a given subset at all nodes beyond or incompatible with
$\sigma$ or it's not.
So def(L$_{1_{\sigma}}$) = \{0, $1_{\sigma}$\}.

Let T be \{$1_{\sigma} \mid \sigma \in {\it P} $\} (in {\bf V}). Let
$\hat{T}$ be a Kripke subset of T,
that is, $\hat{T}_{\sigma} \subseteq$ T and the transition functions are
all the identity. Let $\hat{T}_{0}$ be such that $(\hat{T}_{0})_{\sigma} =
\hat{T}_{\sigma} \; \cup$ \{0\}. $\hat{T}_{0}$ is
an ordinal: T is a set of ordinals, so $\hat{T} \subseteq$ T is also, as is
$\hat{T}_{0}$. So $\hat{T}_{0}$ is a set of transitive sets. It is also
transitive
itself: if $\tau \models x \in y \in \hat{T}_{0}$, then either $\tau \models$ y
= 0,
which contradicts $\tau \models x \in y$, or $\tau \models y \in \hat{T}$, i.e.
$\tau \models y = 1_{\sigma}$ for some $\sigma$, and $\tau \models$ x = 0, and 0
$\in \hat{T}_{0}$.

\begin{lemma}

L$_{\hat{T}_{0}}$ = $\hat{T}_{0}$ \label{T-hat-0}
\end{lemma}

\begin{proof}
L$_{\hat{T}_{0}} $ = $ \bigcup_{\alpha \in \hat{T}_{0}} def(L_{\alpha})\\
$ = $[\bigcup_{\alpha \in \hat{T}} def(L_{\alpha})] \cup def(L_{0}) \\
$ = $[\bigcup_{1_{\sigma} \in \hat{T}} def(L_{1_{\sigma}})] \cup \{0\} \\
$ = $ \bigcup_{1_{\sigma} \in \hat{T}} \{0, 1_{\sigma}\} \cup \{0\} \\
$ = $ \hat{T}_{0} $
\end{proof}

\begin{corollary} $\hat{T}$ is definable over L$_{\hat{T}_{0}}$.
\end{corollary}

\begin{proof}
$\hat{T}$ = $\hat{T}_{0} - \{0\}$ = \{x $\mid$ x $\neq$ 0\} as a definition
over
$\hat{T}_{0}$,
since, for all $\sigma, \bot \models 1_{\sigma} \neq 0$.
\end{proof}

\begin{corollary}

Suppose \{$\hat{T}_{i} \mid i \in$ I\} is a collection of Kripke subsets of T. Then
\{$\hat{T}_{i0} \mid
i \in$ I\} $\cup \bigcup$ \{$\hat{T}_{i0} \mid i \in$ I\} is an ordinal, say
$\xi$, and $\forall i \;
\hat{T}_{i} \in L_{\xi}$. \label{subsets}
\end{corollary}

Now suppose that {\it P} (the underlying partial order of the Kripke structure 
to be built) is a tree of height $\omega$. Also assume that {\it P}
is nowhere degenerate, meaning that each node has incompatible extensions.
Much of what follows would work in a more general context, but we have no
need of such a detailed investigation.
We would like some internal notion of a branch through {\it P}; that is, a
Kripke set that behaves as such. Given an
external branch B (that is, a branch through {\it P} in the classical sense), 
the obvious internalization
of B would be $\{1_{\sigma} \mid \sigma \in B\}$. But if $\tau$ is
perpendicular to any member of B, then at node $\tau$ this looks like \{1\}.
Even worse, $\neg \neg$ B = \{1\}, and it's not possible to get two such
branches forced at any node to be unequal. Hence we would like branches that
reappear even after you've fallen off of them.

In the following, we explicitly distinguish between those definitions and lemmas 
that of necessity are to be evaluated in {\bf V} (labeled ``external"), because they
refer to non-Kripke objects, and those, labeled ``internal", that can be 
evaluated in an appropriate Kripke structure (meaning one containing certain
parameters or satisfying certain exioms from IZF).

\begin{definition} (external) $\hat{P}$ is the full Kripke subset of T;
that is, $\hat{P}_{\sigma}$ = T.
\end{definition}

The following Kripke-internal definition is $\Delta_{0}$, and hence can be
evaluated and applied (such as in the definition and lemma thereafter) in any
transitive Kripke set containing $\hat{P}$ even
without an ambient Kripke universe. (Cf. the comments on ordinality at the end of section
\ref{section-int-def}.)

\begin{definition} (internal)
B $\subseteq$ $\hat{P}$  is a branch through $\hat{P}$  if
\begin{enumerate}
\item	$\alpha \supseteq \beta \in B \rightarrow \alpha \in B$,
\item	$\forall \alpha, \beta \in B \;\;\; \alpha \subseteq \beta \:\vee\: \beta
\subseteq
\alpha,$ and
\item	$\forall \gamma \in \hat{P} \:((\forall \beta \in B \;\; (\gamma
\subseteq \beta
\:\vee\:
\beta \subseteq \gamma)) \rightarrow \gamma \in B)$.
\end{enumerate} \label{branches}
\end{definition}

The reader may have wondered why 1$_{\sigma}$ was taken as being 1 at nodes 
incompatible with $\sigma$, instead of 0, which might first have come to mind.
The answer is clause 1) in the definition above. Otherwise, if we had defined 
(1$_{\sigma}$)$_{\tau} = \emptyset$ for $\tau \perp \sigma$, then consider
what would hold of a branch, i.e. let $\rho \models$ ``B is a branch through $\hat{P}$".
By clause 3) $\rho \models ``B \not = \emptyset"$, so let $\sigma$ be such that
$\rho \models ``1_{\sigma} \in B"$. In practice it will be easy to extend $\rho$ to $\tau
\perp \sigma$; then $\tau \models ``1_{\sigma} \in B".$ $\tau \models ``1_{\sigma} = 0"$ 
because P is a tree (i.e. $\forall \tau' \geq \tau \; \tau' \perp \sigma$), so $\tau 
\models ``0 \in B".$ By clause 1) then $\tau \models ``B = \hat{P}"$, which would
then run afoul of clause 2). So in order to have branches at all, 1$_{\sigma}$ needed
to be defined as it was.

\begin{definition} (external) If $\sigma \models$ U $\subseteq$ $\hat{P}$
then for $\tau
\geq
\sigma$ ext$_{\tau}$(U) = \{$\rho \geq \tau \mid \tau \models 1_{\rho} \in$
U\}: ``the externalisation of U at $\tau$''.
\end{definition}

\begin{lemma} (external)
$\sigma \models$ ``B is a branch through $\hat{P}$'' iff for all $\tau \geq
\sigma$
ext$_{\tau}$(B) is a
branch through \it{P}$_{\geq \tau}$ and 1$_{\bot} \in B_{\sigma}$.
\label{externalization lemma}
\end{lemma}
\begin{proof}
Recall that the notation Tr$_{\geq \tau}$ for a tree Tr means the subtree
of nodes extending $\tau$ (including $\tau$ itself) and the notation
$X_{\sigma}$ for a Kripke set X refers to the collection of elements forced
at $\sigma$ into X. Recall also that a branch of a tree classically is a
non-empty, linearly ordered subset of a tree which is closed downwards and
(if the tree has no terminal leaves, as in our context) has no final
element.

$\rightarrow$ : Suppose $\tau \geq \sigma$.

a) ext$_{\tau}$(B) is non-empty: For $\tau = \bot \; \bot \models ``\forall
\beta \in \hat{P} \; \beta \subseteq 1_{\bot}$''. By clause 3) in the
definition of a branch, $\bot \models ``1_{\bot} \in$ B'', so $\bot \in$
ext$_{\bot}$(B). For $\tau \not = \bot \; \tau \models ``\forall \beta \in
\hat{P} \; \beta = 1 \vee \beta \subseteq 1_{\tau}$''. Again by clause 3),
$\tau \models ``1_{\tau} \in$ B'', and $\tau \in$ ext$_{\tau}$(B).

b) ext$_{\tau}$(B) is linearly ordered: Suppose $\rho, \xi \in$
ext$_{\tau}$(B). Then $\tau \models ``1_{\rho}, 1_{\xi} \in B$''. By clause
2) in the definition of a branch, $\tau \models ``1_{\rho} \subseteq
1_{\xi} \; \vee \; 1_{\xi} \subseteq 1_{\rho}$''. If $\rho$ and $\xi$ were
incompatible, then $\rho \not \models ``1_{\xi} \subseteq 1_{\rho}$'' and
$\xi \not \models ``1_{\rho} \subseteq 1_{\xi}$''. This, however,
contradicts the semantics of ``$\vee$'' in Kripke structures, by which one
of the two disjuncts must hold. Hence $\rho$ and $\xi$ are compatible.

c) ext$_{\tau}$(B) is closed downwards: Suppose $\rho \in$ ext$_{\tau}$(B)
and $\tau \leq \xi \leq \rho$. Then $\tau \models ``1_{\rho} \in B$'' and
$\tau \models ``1_{\rho} \subseteq 1_{\xi}$''. By clause 1) in the
definition of a branch, $\tau \models ``1_{\xi} \in B$'', and $\xi \in$
ext$_{\tau}$(B).

d) ext$_{\tau}$(B) has no final element: Suppose to the contrary that
$\rho$ is the final element of ext$_{\tau}$(B). Consider ext$_{\rho}$(B).
From the proof of part a) above, $\rho \in$ ext$_{\rho}$(B). If
ext$_{\rho}$(B) = $\{\rho\}$, then let $\pi$ be an immediate successor of
$\rho$ (which exists since {\it P} was assumed to have an ordinal height
and no degenerate nodes; see the comments after corollary \ref{subsets}). 
Again from part a), we have already seen that $\pi \models
``\forall \beta \in \hat{P} \; \beta = 1 \vee \beta \subseteq 1_{\pi}$'';
since branches are, by definition, subsets of $\hat{P}$, we have in
particular $\pi \models ``\forall \beta \in B \; 1_{\pi} \subseteq \beta
\vee \beta \subseteq 1_{\pi}$''. At an extension $\xi$ of $\rho$
incompatible with $\pi$, $\xi \models ``1_{\pi} = 1$'', so $\xi \models
``\forall \beta \in B \; \beta \subseteq 1_{\pi}$''. Finally, at $\rho$
itself, the only element forced at $\rho$ into B is, by hypothesis,
1$_{\rho}$, and $\rho \models ``1_{\pi} \subseteq 1_{\rho}$''. To
summarize, $\rho \models ``\forall \beta \in B \; 1_{\pi} \subseteq \beta
\vee \beta \subseteq 1_{\pi}$''. By clause 3), $\rho \models ``1_{\pi} \in
B$'', and $\pi \in$ ext$_{\rho}$(B), contrary to hypothesis. So
ext$_{\rho}$(B) $\not = \{\rho\}$.

Let $\pi \not = \rho$ be some member of ext$_{\rho}$(B). We will show that
$\pi \in$ ext$_{\tau}$(B), contradicting the choice of $\rho$ as the final
element of ext$_{\tau}$(B). This argument, once again, uses clause 3). Our
goal is to show that $\tau \models ``\forall \beta \in B \; 1_{\pi}
\subseteq \beta \vee \beta \subseteq 1_{\pi}$''. Once that is accomplished,
it follows that $\tau \models ``1_{\pi} \in B$'', which means $\pi \in$
ext$_{\tau}$(B).

So let $\xi \geq \tau$. If $\xi$ and $\rho$ are incompatible, then so are
$\xi$ and $\pi$ ($\pi$ extends $\rho$). So $\xi \models ``1_{\pi} = 1$'',
and 1 is a superset of every member of $\hat{P}$. If $\xi$ and $\rho$ are
compatible, then either $\rho \leq \xi$ or $\tau \leq \xi < \rho$. In the
first case, $\xi \models ``1_{\pi} \in B$'', because $\rho$ models the
same. By clause 2) in the definition of a branch, $\xi \models ``\forall
\beta \in B \; 1_{\pi} \subseteq \beta \vee \beta \subseteq 1_{\pi}$''. In
the second case, $\tau \leq \xi < \rho$, we have in any case $\xi \models
``1_{\rho} \in B$'', because $\tau$ models the same. If $\xi \models
``\beta \in B$'', then by clause 2) $\xi \models ``1_{\rho} \subseteq \beta
\vee \beta \subseteq 1_{\rho}$''. If $\xi \models ``1_{\rho} \subseteq
\beta$'' then $\xi \models ``1_{\pi} \subseteq \beta$'', because $\xi
\models ``1_{\pi} \subseteq 1_{\rho}$''. If, on the other hand, $\xi
\models ``\beta \subseteq 1_{\rho}$'', then $\rho \models ``\beta \subseteq
1_{\rho}$''. That implies that $\beta$ = 1$_{\mu}$ for some $\mu \geq
\rho$. Furthermore, by clause 2), $\rho \models ``1_{\pi} \subseteq
1_{\mu}$'' or $\rho \models ``1_{\mu} \subseteq 1_{\pi}$''. In the first
case we have $\pi \geq \mu$, in the second $\mu \geq \pi$, and in either
case $\xi \models ``1_{\pi} \subseteq 1_{\mu}$'' or $\xi \models ``1_{\mu}
\subseteq 1_{\pi}$''.

Finally, for this direction of the proof it remains to show that 1$_{\bot}
\in B_{\sigma}$. $\bot \models ``\forall \beta \in \hat{P} \; \beta
\subseteq 1_{\bot}$''. By clause 3), $\sigma \models ``1_{\bot} \in B$''.

$\leftarrow :$
\begin{enumerate}
\item	$\sigma \models ``\alpha \supseteq \beta \in B \rightarrow \alpha \in
B$'': Suppose $\sigma \models ``1_{\tau} \supseteq 1_{\rho} \in B$''. If
$\tau \bot \sigma$ or $\tau < \sigma$ then $\sigma \models ``1_{\tau} = 1 =
1_{\bot} \in B$''. Otherwise $\tau \geq \sigma$. Then $\tau \leq \rho$,
because otherwise $\tau \models ``1_{\rho} = 1$'' and $\tau \not \models
``1_{\tau} = 1$''. Since $\rho \in$ ext$_{\sigma}$(B), which is by
hypothesis a branch, and hence closed downwards, we have $\tau \in$
ext$_{\sigma}$(B), i.e. $\sigma \models ``1_{\tau} \in B$''.

\item	$\sigma \models ``\forall \alpha, \beta \in B \;\;\; \alpha \subseteq
\beta \:\vee\: \beta \subseteq \alpha$'': Suppose $\sigma \models
``1_{\tau}, 1_{\rho} \in B$''. If $\tau \bot \sigma$ or $\tau < \sigma$
then $\sigma \models ``1_{\tau} = 1 \supseteq 1_{\rho}$''. Similarly if
$\rho \bot \sigma$ or $\rho < \sigma$. So $\tau, \rho \geq \sigma$. By
hypothesis $\tau, \rho \in$ ext$_{\sigma}$(B), which is a branch, and hence
linearly ordered; therefore $\tau \leq \rho$ or $\rho \leq \tau$, implying
$\bot \models ``\rho \subseteq \tau$'' or $\bot \models ``\tau \subseteq
\rho$'' respectively.

\item	$\sigma \models ``\forall \gamma \in \hat{P} \:(\forall \beta \in B
\;\; \gamma \subseteq \beta \:\vee\: \beta \subseteq \gamma \rightarrow
\gamma \in B)$'': Suppose $\sigma \models ``\forall \beta \in B \;\;
1_{\tau} \subseteq \beta \:\vee\: \beta \subseteq 1_{\tau}$''. If $\tau
\bot \sigma$ or $\tau < \sigma$ then $\sigma \models ``1_{\tau} = 1 =
1_{\bot} \in B$''. Otherwise $\tau \geq \sigma$. If $\tau$ were
incompatible with some $\rho \in$ ext$_{\sigma}$(B), then $\tau \not
\models ``1_{\rho} \subseteq 1_{\tau}$'' and $\rho \not \models ``1_{\tau}
\subseteq 1_{\rho}$'', yielding $\sigma \not \models ``1_{\tau} \subseteq
1_{\rho} \:\vee\: 1_{\rho} \subseteq 1_{\tau}$'', contrary to the
hypothesis. Since ext$_{\sigma}$(B) as a branch is a maximal linearly
ordered subset of {\it P}$_{\geq \sigma}$, $\tau \in$ ext$_{\sigma}$(B),
and $\sigma \models ``1_{\tau} \in B$''.
\end{enumerate}
\end{proof}

\subsection{Resolvability + IKP + V=L + $\neg\Sigma_{1}$ DC}
\label{section-not-DC}

To be perfectly clear, we really should state what Resolvability means in this context.
The standard definition, that there is a $\Delta_{1}$-definable function  f
: ORD $\rightarrow$ {\bf V}  such that {\bf V} = $\bigcup$ rng(f), is
perfectly coherent, but could be contested as not being the appropriate
correlate to classical resolvability in an intuitionistic setting. After
all, maybe the linearity of the domain is what's vital, so a resolution
should have domain some linearly ordered set of ordinals,
presumably a $\Delta_{1}$ definable set of ordinals. We will take the latter as
the definition of a resolution, because it seems to be a more difficult
definition to fit
(although it is not clear that a function defined on a $\Delta_{1}$ set of
ordinals could
be extended to all the ordinals), and point out that if {\bf V} = {\bf L} then
$\alpha \mapsto L_{\alpha}$ shows resolvability according to the former notion of a
resolution. In fact, the resolution we will be offering will even be extendable to a
$\Delta_{1}$ function on the whole model.

The underlying partial order of the model will be $2^{<\omega}$, with the order
being end-extension ($\sigma < \tau$ iff $\tau$ is a proper end-extension of
$\sigma$.). For each real R : $\omega \rightarrow$ 2 there is a canonical
branch
B(R)
through T:
$$
ext_{\tau}(B(R)) = \{\rho \geq \tau \mid \forall j \; length(\tau) \leq j <
length(\rho) \rightarrow
\tau(j) = R(j) \},
$$
or, equivalently,
$$
B(R)_{\tau} = \{ 1_{\rho} \mid \rho < \tau \; \vee \; \rho \perp \tau \;
\vee \; (\rho \geq \tau \; \wedge \; \forall j \; length(\tau) \leq j <
length(\rho) \rightarrow \tau(j) = R(j)) \}.
$$
B(R) is a branch by the preceding lemma. Notice that if R$_{0}$ and R$_{1}$
differ cofinally then $\bot \models$ B(R$_{0}) \neq$ B(R$_{1}$).

Start with a model {\bf V} of classical ZF. Let G := \{G$_{i} \mid i \in \omega$\} 
be a set of reals mutually Cohen-generic over {\bf V} (that is, reals generic for the forcing
partial order 2$^{< \omega}$).
Actually, the rest of this section save the last lemma applies just as well to any set of 
(mutually cofinally differing) reals. (This last fact we will use
in section \ref{not Pi_2}, where many constructions and lemmas similar to those below are
used, often with the offered proof being merely a reference to the corresponding proof here,
even though the real there is not even generic.)
In {\bf V}[G], construct the (Kripke) branches B(G$_{i}$), $i \in \omega$.
For convenience, call the branch B(G$_{i}$) just B$_{i}$. Consider
\{B$_{i} \mid i
\in \omega$\} $\cup$ \{$\hat{P}$\}. Let $\xi$ be the (Kripke) ordinal from corollary
\ref{subsets} applied to
this set; in particular, each B$_{i}$ and $\hat{P}$ is a member of L$_{\xi}$.

\begin{lemma}
\{B$_{i} \mid i \in \omega$\} is definable over L$_{\xi}$, as the set of
branches through $\hat{P}$.

\end{lemma}
\begin{proof}
L$_{\xi}$ = $\bigcup_{\alpha \in \hat{P}_{0}}$ def(L$_{\alpha}$) $\cup
\bigcup_{i \in
\omega}$
def(L$_{B_{i0}}$) $\cup$ def(L$_{\hat{P}_{0}}$). The first union is just
$\hat{P}_{0}$, by
lemma
\ref{T-hat-0}, which doesn't contain a branch. Each term in the second union,
def(L$_{B_{i0}}$), is just def(B$_{i0}$), again by \ref{T-hat-0}. No branch
contains 0,
because, if it did, by clause 1) in the definition of a branch, it would then
contain all of $\hat{P}$, contradicting clause 2). So
any branch which is a subset of B$_{i0}$ is also a subset of B$_{i}$. But
B$_{i}$ is itself a
branch, and no proper subset of a branch is a branch, because of clause 3), as follows. Suppose B is a branch, $\hat{B} \subseteq$ B, and $\hat{B}$ is also a
branch. If $\gamma \in$ B and $\beta \in \hat{B}$, $\beta \in$ B too; using
that
B is a branch, $\gamma \subseteq \beta$ or $\beta \subseteq \gamma$ by
clause 2). Since $\hat{B}$ is a branch, $\gamma \in \hat{B}$ by clause 3).
So B $\subseteq \hat{B}$, and B = $\hat{B}$. In the final case,
def(L$_{\hat{P}_{0}}$),
or, more simply, again by \ref{T-hat-0}, def($\hat{P}_{0}$), consider any
definition, say $\phi(x)$, over $\hat{P}_{0}$. It contains only finitely
many parameters 1$_{\sigma}$. Go to a node $\tau$ beyond or incompatible
with all such $\sigma$'s. $\tau \models \phi (1_{\rho})$ for some extension
$\rho$ of $\tau$  iff $\tau \models \phi (1_{\rho})$ for all extensions
$\rho$ of $\tau$ of the same length, by symmetry. (There is an automorphism
of $\hat{P}_{0}$ switching any given pair of $\rho$'s.) So $\phi$ does not
define a branch.
\end{proof}

The model {\bf M} can now be defined inductively:
$$
{\bf M}_{0} = L_{\xi};
$$
$$
{\bf M}_{\alpha + 1} = def({\bf M_{\alpha}});
$$
$$
{\bf M_{\lambda}} = \bigcup_{\alpha < \lambda} {\bf M_{\alpha}}; \; and
$$
$$
{\bf M} = {\bf M}_{\omega_{1}^{CK}}.
$$

\begin{lemma}
{\bf M} $\models$ IKP. \label{IKP}
\end{lemma}
\begin{proof}
Most of the axioms depend on the fact that if $\beta > \alpha$ then {\bf
M}$_{\beta}$
is an end-extension of {\bf M}$_{\alpha}$. For instance, as has already been
mentioned,
{\bf M}$_{\alpha}$ is proven to be extensional inductively on $\alpha$, using
this
end-extension property for the limit cases. As another example, consider
Pairing.
If y, z $\in$ {\bf M}, let $\alpha$ be such that y, z $\in$ {\bf M}$_{\alpha}$.
Now
consider [``x = y $\vee$ x = z''] $\in$ {\bf M}$_{\alpha + 1}$. This set is the
pair
\{y, z\} in {\bf M}$_{\alpha + 1}$; since {\bf M} end-extends {\bf
M}$_{\alpha +
1}$,
the same fact holds in {\bf M}. Empty and Union are similar. $\Delta_{0}$
Comprehension
also depends on this, being the reason $\Delta_{0}$ formulas are absolute
between
{\bf M}$_{\alpha}$ and {\bf M}. (So if $\phi$(x) is $\Delta_{0}$, and all of its
parameters, as well as X, are in {\bf M}$_{\alpha}$, then [x $\in$ X $\wedge$
$\phi$(x)] $\in$ {\bf M}$_{\alpha + 1}$ is the desired subset in {\bf M}.) Even
Foundation uses end-extensionality. Suppose {\bf M} $\not\models$ Foundation.
Let $\phi$ and $\sigma$ be such that $\sigma \models \forall x ((\forall y \in x
 \; \phi(y)) \rightarrow \phi(x))$ and $\sigma \not\models \forall x \;
\phi(x)$. Let
$\tau \geq \sigma$ and $x \in ({\bf M}_{\alpha})_{\tau}$ be such that $\tau
\not\models \phi(x)$; moreover, assume that $\alpha$ is the least such ordinal.
By the choice of $\alpha$, if y $\in ({\bf M}_{\beta})_{\xi} \; (\beta <
\alpha)$
then $\xi \models \phi(y)$. By end-extensionality, if $\xi \models y \in x$ then
y $\in {\bf M}_{\beta}$, some $\beta < \alpha$. So $\tau \models \forall y
\in x
\;
\phi(y)$, hence, by hypothesis, $\tau \models \phi(x)$, a contradiction. So
{\bf M} $\models$ Foundation.

For $\Delta_{0}$ Bounding, {\bf M} can be considered to be
constructed by a $\Sigma_{1}$ induction over the standard
L$_{\omega_{1}^{CK}}$[G], so any $\Delta_{0}$({\bf M}) formula $\phi$ is
equivalent
to a $\Delta_{0}$(L$_{\omega_{1}^{CK}}$[G]) formula $\phi^{\ast}$. (This
would be
proven inductively on formulas, with, for instance, $\phi \rightarrow \psi$
going to $\forall
\sigma \in 2^{<\omega} \sigma \models \phi \rightarrow \sigma \models
\psi$. The
important observation is that only bounded quantifiers are added.) So if {\bf
M} $\models ``\forall x
\in A \; \exists y \; \phi (x, y)$'' then in L$_{\omega_{1}^{CK}}$[G]
$\forall x \in_{\bf M}
A \;\exists \alpha \; \exists y \in {\bf M_{\alpha}} \; \phi^{\ast} (x,
y)$, and
by the
admissibility of L$_{\omega_{1}^{CK}}$[G] the $\alpha$'s needed can be
bounded.

{\bf M} also satisfies Infinity. To see this, define
inductively on n $\in \omega$ the set $\overline{n} \in$ {\bf M}, the internal
version of n, as follows: $\overline{0} = \emptyset^{\bf M}$, $\overline{n+1} \;
= \; (n \cup \{n\})^{\bf M}$. Inductively, $\overline{n} \in {\bf M}_{n}$; in
fact, $\overline{n+2} \in {\bf M}_{n}$. Definably over {\bf M}$_{\omega}$, let
$\overline{\omega}$ be \{$\alpha \mid \alpha$ is an ordinal, and either $\alpha$
= 0 or $\alpha$ is a successor, and $\forall \beta \in \alpha \; \beta$ is 0 or
a
successor\}. $\overline{\omega}$ witnesses that {\bf M} $\models$ Infinity;
moreover, one can show inductively that $\overline{\omega} \cap {\bf M}_{n} \;
= \; \{\overline{m} \mid m \leq n+2 \}$, so $\overline{\omega} = \{ \overline{n}
\mid n \in \omega \}$.
\end{proof}

\begin{lemma}
{\bf M} is resolvable. \label{resolvable}
\end{lemma}
\begin{proof}
The resolution of {\bf M} will be along $\overline{\omega_{1}^{CK}}$, which is
meant as the
internalization of $\omega_{1}^{CK}$, as $\overline{\omega}$ is for $\omega$.
$\overline{\omega_{1}^{CK}}$ is defined (in {\bf V}[G]) as \{ $\overline{\alpha} \mid \alpha <
\omega_{1}^{CK}$ \}, where $\overline{\alpha} \in$ {\bf M} is
\{ $\overline{\beta} \mid \beta < \alpha $ \}. We need to show that for $\alpha
< \omega_{1}^{CK} \; \overline{\alpha} \in$ {\bf M}. Say that an {\bf M}-ordinal
is standard if it's 0, the successor of a standard ordinal, or an
$\overline{\omega}$-
limit of standard ordinals. Notice that this definition is internal to {\bf M}.
By
the Second Recursion Theorem, there is a $\Sigma$ definition of standardness in
{\bf M}. Inductively on $\alpha \geq \omega$, the standard ordinals in {\bf
M}$_{\alpha}$
are a subset of $\overline{\alpha}$. To show that for $\alpha <
\omega_{1}^{CK},
\;
\overline{\alpha} \in {\bf M}$, assume inductively that for $\beta < \alpha \;
\overline{\beta} \in {\bf M}$. By the admissibility of the ambient universe
(i.e.
L$_{\omega_{1}^{CK}}$[G]), there is a $\gamma$ such that for each $\beta <
\alpha$ both $\overline{\beta}$ and a witness that $\overline{\beta}$ is
standard
are in {\bf M}$_{\gamma}$. (Here, a witness to standardness is a witness to the
$\Sigma$ definition of standardness.) Then either $\overline{\alpha}$ is
definable
over {\bf M}$_{\gamma}$ as the set of standard ordinals, or $\overline{\alpha}$
is already in {\bf M}$_{\gamma}$. Hence $\overline{\omega_{1}^{CK}} \subseteq$
{\bf M}.

In order to have a resolution along $\overline{\omega_{1}^{CK}}$, we need it
$\Delta_{1}$
definable. It's already $\Sigma$ definable, as the set of standard
ordinals. But
the
witnesses to standardness are not really far away from $\overline{\alpha}$
itself.
That is, classically, a witness to the standardness of $\alpha$, consisting of
suitable
predecessors and cofinal $\omega$-sequences, could be found in L$_{\alpha +
\omega}$.
So to see whether $\overline{\alpha}$ is standard in {\bf M}, look in
L$_{\overline{\alpha}
+ \overline{\omega}}$. The function $\alpha \mapsto L_{\alpha}$ is total in
{\bf M}$\models$ IKP (see \cite{me}), so the $\Sigma_{1}$ definition of
``$\alpha$
is not standard'' is ``there is a set L$_{\alpha + \overline{\omega}}$ and
there
is no witness
in that set that $\alpha$ is standard''. The $\Sigma$ definition of standardness
is equivalent to a $\Sigma_{1}$ formula. Now we need to check that this suffices for
the notion of $\Delta_{1}$ defined earlier. Recall that toward the beginning of
section \ref{secIKP} a $\Delta_{1}$ property of a variable $x$ was defined as being given
by a pair of $\Delta_{0}$ formulas $\phi(x,y), \psi(x,y)$ with free variables $x,y$ such that
$$\forall x [\exists y \neg \neg (\phi(x,y) \vee \psi(x,y)) \; \wedge
\forall y \neg (\phi(x,y) \wedge \psi(x,y))].
$$
So let $\phi(x,y)$ be
$$x \in ORD \; \wedge \; y = \langle z_{0}, z_{1} \rangle \; \wedge \; z_{1} \; witnesses \; 
that \; z_{0} = {\bf L}_{x + \overline{\omega}}
$$
$$\wedge \; \exists w \in z_{0} \; w \; witnesses \; that \; x \; is \; standard,
$$ 
and $\psi(x,y)$ be
$$(x \not \in ORD) \; \vee \; (x \in ORD \; \wedge \; y = \langle z_{0}, z_{1} \rangle \; 
\wedge \; z_{1} \; witnesses \; that \; z_{0} = {\bf L}_{x + \overline{\omega}}
$$
$$\wedge \; \forall w \in z_{0} \; w \; does \; not \; witness \; that \; x \; is \; standard),
$$ 
where ``$x + \alpha$" is defined inductively on $\alpha$ even for non-ordinals $x$ in the 
obvious way.
For ordinals $x$ there is certainly a $y$ as desired, and it is also clear that
$\forall x \; \neg \neg (x \in ORD \; \vee \; x \not \in ORD)$. The challenge presented by the 
formalism
as set up is that the $y$ must be chosen before it is decided whether $x$ is an ordinal or not.
Here we have to use the development from the penultimate paragraph of section 
\ref{section-int-def}, in which iterations of definability along well-founded (and not
necessarily well-ordered) lengths were developed. Let 
$z_{0}$ be $(x + \overline{\omega})$-def({\bf 0}) and $z_{1}$ an appropriate construction
of such. If $x$ ever turns out to be an ordinal, then $z_{0}$ is automatically 
${\bf L}_{x + \overline{\omega}}$. With this $\phi$ and $\psi$, the notion ``$\alpha$ is 
standard'' is shown to be $\Delta_{1}$ definable.

From here, the resolution of {\bf M} is its very definition. L$_{\xi}$ is needed
as a parameter, which is no problem because it exists as a set in {\bf M}:
L$_{\xi}$
is definable over {\bf M}$_{0}$ = L$_{\xi}$ as \{ x $\mid$ x = x \}, and so is
in
{\bf M}$_{1}$. ``Y = $\bigcup$X'' is a $\Sigma$ relation, as is ``Y = def(X)''
(for details on the latter, see \cite{me}). So $\alpha \rightarrow {\bf
M}_{\alpha}$
is definable via $\Sigma$ recursion as a $\Sigma$ relation, hence $\Sigma_{1}$
also.
Any $\Sigma_{1}$ definable function is $\Delta_{1}$, even with the notion of
$\Delta_{1}$ operable here, for the same reason as in the classical context, as
follows. Suppose the relation f(x) = y is definable as $\exists z \; \phi(x,y,z).$
Then let $\psi(x,y,w,z)$ be $\phi(x,w,z) \; \wedge \; y \not = w$. Then indeed
$$\forall x,y [\exists w,z \neg \neg (\phi(x,y,z) \vee \psi(x,y,w,z)) \; \wedge
\forall w,z \neg (\phi(x,y,z) \wedge \psi(x,y,w,z))].
$$
Hence {\bf M} is resolvable.
\end{proof}

\begin{lemma}
{\bf M} $\models$ V=L. \label{L}
\end{lemma}
\begin{proof}
First observe that L$^{\bf M} \models$ IKP. Extensionality and Foundation hold,
because L$^{\bf M}$ is a definable sub-class of a model of Extensionality and
Foundation. Pairing is simple: if x $\in L_{\alpha}$ and y $\in L_{\beta}$ then
x, y $\in L_{\alpha \cup \beta}$, and \{x, y\} $\in$ def(L$_{\alpha \cup
\beta}$) =
L$_{(\alpha \cup \beta) + 1}$. Union is similar: if x $\in L_{\alpha}$ then
$\bigcup x \in$ def(L$_{\alpha}$) = L$_{\alpha + 1}$. $\emptyset \in$
def(L$_{0}$)
= L$_{1}$, and $\overline{\omega} \in$ def(L$_{\overline{\omega}}$) =
L$_{\overline{\omega}
+ 1}$ to satisfy Infinity; for details on the latter, see \cite{me}. For
$\Delta_{0}$
Comprehension, consider \{ x $\in$ X $\mid \phi$(x) \}, $\phi$ a $\Delta_{0}$ formula. Let
$\alpha$ be such that L$_{\alpha}$ contains X as well as $\phi$'s parameters.
Then
\{ x $\in$ X $\mid \phi$(x) \} $\in$ def(L$_{\alpha}$) = L$_{\alpha + 1}$. Even
$\Delta_{0}$ Bounding is smooth. Suppose $\forall x \in X \; \exists y \in
L^{\bf M} \;
\phi(x, y),\; \phi \Delta_{0}$. In {\bf M}, $\forall x \in X \; \exists \alpha
\; \exists
y \in L_{\alpha} \; \phi(x, y)$. By $\Sigma$ Bounding in {\bf M}, let A be
a set
of
ordinals such that $\forall x \in X \; \exists \alpha \in A \; \exists y \in
L_{\alpha} \; \phi(x, y)$. Let $\beta$ = A $\cup \bigcup$A. $\forall x \in X \;
\exists
y \in L_{\beta} \; \phi(x, y)$, and L$_{\beta} \in$ L$^{\bf M}$, showing
$\Delta_{0}$
Bounding.

For any $\alpha < \omega_{1}^{CK}$, since $\overline{\alpha} \in$ {\bf M},
L$_{\overline{\alpha}} \in$ {\bf M}.
$\overline{\alpha}$ is definable over L$_{\overline{\alpha}}$, so
$\overline{\alpha} \in$ L$^{\bf M}$. Also, \{B$_{i} \mid i
\in \omega$\} $\in$ {\bf M}, as is $\hat{P}$, so $\xi \in$ {\bf M}, and
L$_{\xi} \in$
L$^{\bf M}$.
Finally, induction shows that any admissible set containing L$_{\xi}$ as
well as
each $\overline{\alpha} < \overline{\omega_{1}^{CK}}$ must contain all of {\bf
M}.
The crucial fact needed in this induction is that the function X $\mapsto$
def(X)
is $\Delta_{1}$ definable, provably in IKP. This was essentially done in
\cite{me},
even though it was not flagged as such. There it was shown, in IZF, that this
function
in $\Delta_{1}$ over all $\beta > \alpha_{aug} + 7$, where X $\in
L_{\alpha_{aug}}$
and $\alpha_{aug}$ is defined inductively as $\bigcup \{ \gamma_{aug} \mid
\gamma
\in \alpha \} \cup \{ \alpha \} \cup (\omega + 1)$. (The purpose of this
definition
is to get $\omega$ into $\alpha$, hereditarily, as the carrier of syntax.) The
proof
is that a witness to ``Y = def(X)'' is definable over L$_{\alpha_{aug} + 7}$, so
IZF isn't necessary; all you need is some ordinal containing as a member
$\alpha_{aug} + 7$.
Moreover, the recursion theorem shows that in IKP + Infinity the function
$\alpha
\mapsto \alpha_{aug} + 7$ is total. So in IKP, X $\mapsto$ def(X) is
$\Delta_{1}$.
This allows the inductive proof that {\bf M}$_{\alpha} \in$ L$^{\bf M}$, $\alpha
\in \omega_{1}^{CK}$.
\end{proof}

\begin{lemma}
{\bf M} $\models \neg \Sigma_{1}$ DC. \label {not-sig}
\end{lemma}
\begin{proof}
This is where we will use the particular choice of the G$_{i}$'s. Consider the
relation
$\phi$(X, Y)
``if X is an n-tuple of distinct elements of \{B$_{i} \mid i \in \omega$\} then
Y is an
end-extension of X also of distinct elements''. Any function f with domain
$\omega$
satisfying $\phi$(f(i), f(i+1)) (and starting from $\emptyset$) would easily
produce an
$\omega$-sequence through \{G$_{i} \mid i \in \omega$\} of distinct
elements, of
which
there are none in {\bf V}[G], much less L$_{\omega_{1}^{CK}}$[G], to say nothing of
{\bf M}.
\end{proof}

\subsection{$\Pi_{2}$-Reflection + IKP + V=L + $\neg \Sigma_{1}$DC +
$\neg$Resolvability}

Recall the definition of a resolution (see section \ref{section-not-DC}) as being 
a $\Delta_{1}$-definable function f : O$\rightarrow${\bf V} such that {\bf V} = 
$\bigcup$rng(f) and O $\subseteq$ ORD is $\Delta_{1}$-definable and linearly ordered.
In the last section, the argument works just as well for this definition as for
the classical one (O = ORD, of course not mentioning the requirement of the linear
order). Here, though, it wouldn't. If {\bf V} = L then f($\alpha$) = L$_{\alpha}$ is 
a classical resolution of {\bf V}. Hence the need here for a non-classical notion
of a resolution.

The underlying partial order is 2$^{<\omega}$, as before. 

Let G and H be countable sets of mutual Cohen generics. The ambient universe
will be
$\bigcup_{H_{fin}}$ L$_{\omega_{1}^{CK}}$[G $\cup$ H$_{fin}$], where H$_{fin}$
ranges
over the finite subsets of H. Within this universe the full intuitionistic model
{\bf M}$_{big}$
can be defined inductively, just like forcing: a set at stage $\alpha$ is a
function that, for
each node $\sigma \in 2^{<\omega}$, picks out a collection of sets from stages
$< \alpha$
(and this collection must grow as $\sigma$ grows).

More precisely, suppose {\bf M} is an extensional Kripke model (for the language
of set theory) with underlying partial order {\it P} and transition functions
f$_{\sigma \tau}$ within a classical universe {\bf V}. {\cal P}({\bf M}), the
power set of {\bf M}, can be defined
similarly to the way def({\bf M}) was defined in section \ref{general}. A set in
{\cal P}({\bf M}) at a node $\sigma$ is a function f with domain {\it P}$_{\geq
\sigma}$
such that f($\tau$) $\subseteq$ {\bf M}$_{\tau}$, and if x $\in$ f($\sigma$)
then
f$_{\sigma \tau}$(x) $\in$ f($\tau$). The transition function f$_{\sigma \tau}$
is restriction: f$_{\sigma \tau}$(f) = f $\uparrow$ {\it P}$_{\geq \tau}$. {\bf M} can be
embedded
in {\cal P}({\bf M}), with x $\in$ {\bf M}$_{\sigma}$ being identified with the
function f such that f($\tau$) = \{ y $\in$ {\bf M}$_{\tau} \mid$ y $\in$
f$_{\sigma \tau}$(x) \}. By the extensionality of {\bf M}, this embedding is
1-1. $\sigma \models ``f \in g$'' if f is the image of some x under this
embedding,
and x $\in$ g($\sigma$). {\cal P}({\bf M}) is easily seen to be extensional, and
an end-extension of {\bf M}. Furthermore, this process of taking the power set
can
be iterated transfinitely: {\bf M}$_{0}$ = {\bf M}, {\bf M}$_{\alpha + 1}$
= {\cal P}({\bf M}$_{\alpha}$), and {\bf M}$_{\lambda}$ =
$\stackrel{lim}{\rightarrow}$
\{{\bf M}$_{\alpha} \mid \alpha < \lambda$\} (or, more colloquially, $\bigcup
_{\alpha < \lambda} {\bf M}_{\alpha}$). If $\alpha > \beta$ then {\bf
M}$_{\alpha}$ is an extensional
end-extension of {\bf M}$_{\beta}$.

In the current context, {\bf V} is $\bigcup_{H_{fin}}$ L$_{\omega_{1}^{CK}}$[G
$\cup$ H$_{fin}$].
{\bf V} $\models$ KP, but {\bf V} does $\underline{not}$ satisfy Power Set, so some care
must be taken
in the construction of {\cal P}({\bf M}). We
start with {\bf M} being the empty structure -- {\bf M}$_{\sigma}$ = $\emptyset$
-- but we do not claim that each {\bf M}$_{\alpha}$ is a set. Rather, the
generation
of the {\bf M}$_{\alpha}$-hierarchy is given by a $\Sigma_{1}$ induction, so the
relation ``f $\in$ ({\bf M}$_{\alpha}$)$_{\sigma}$'' is $\Sigma_{1}$. Let
{\bf M}$_{big}$ be {\bf M}$_{\omega_{1}^{CK}}$. The final model {\bf M} is
L$^{{\bf M}_{big}}$; that this definition makes sense (i.e. that $\alpha \mapsto
L_{\alpha}$ is a total $\Delta_{1}$-definable function via the standard
definitions)
is implied by the following lemma.

\begin{lemma}
{\bf M}$_{big}$ $\models$ IKP.
\end{lemma}
\begin{proof}
For Empty, the function f($\sigma$) = $\emptyset$ is in {\bf M}$_{1}$. For
Infinity, defining $\overline{\omega}$ as in the previous section, the function
f($\sigma$) = $\overline{\omega}$ is in M$_{\omega + 1}$. For Pairing,
if x, y $\in$ {\bf M}$_{\alpha}$, let f($\sigma$) in {\bf M}$_{\alpha + 1}$ be
\{x, y\}. For Union, if x $\in$ {\bf M}$_{\alpha}$, then f($\sigma$) =
$\bigcup$x($\sigma$)
is also in {\bf M}$_{\alpha}$. Extensionality was already remarked as holding.
For
$\Delta_{0}$ Comprehension, by end-extensionality a $\Delta_{0}$ formula over
{\bf M}$_{big}$ is equivalent to a $\Delta_{0}$ formula in {\bf V}, and
$\Delta_{0}$
Comprehension in {\bf V} suffices. For Foundation, suppose to the contrary that
$\sigma \models \forall x (\forall y \in x \; \phi(y) \; \rightarrow \;
\phi(x))$,
but $\sigma \not\models \forall x \phi(x)$. Let $\alpha$ be the least ordinal
such
that for some x $\in$ {\bf M}$_{\alpha}$, $\tau \geq \sigma \; \tau \not\models
\phi(x)$. But if $\xi \models y \in x$ then $y \in {\bf M}_{\beta}$, so $\xi
\models \phi(y)$. Hence $\tau \models \forall y \in x \; \phi(y)$, and $\tau
\models \phi(x)$, a contradiction. Finally, for $\Delta_{0}$ Bounding, suppose
$\sigma \models \forall x \in A \; \exists y \; \phi(x, y), \; \phi$ a $\Delta_{0}$ formula.
Then in {\bf V} $\forall \tau \geq \sigma \; \forall x \in A_{\tau} \; \exists
\alpha \; \exists y \in {\bf M}_{\alpha} \; \tau \models \phi(x, y)$. By
$\Delta_{0}$ Bounding in {\bf V}, the $\alpha$s needed can be bounded, say by
$\beta$, and again by $\Delta_{0}$ Bounding in {\bf V} a bounding set B in
{\bf M}$_{\beta + 1}$ can be constructed.
\end{proof}

\begin{corollary}
{\bf M} $\models$ V=L.
\end{corollary}
\begin{proof}
This is the basic result of \cite{me}.
\end{proof}

\begin{lemma}
{\bf M} $\models$ IKP + $\Pi_{2}$ Reflection.
\end{lemma}
\begin{proof}
The minor axioms of IKP are all easy. (For Extensionality, use the fact that
{\bf M}
is a transitive sub-class of the extensional {\bf M}$_{big}$. For Foundation,
use
the definability of {\bf M} inside of the well-founded {\bf M}$_{big}$.)
$\Delta_{0}$ Comprehension follows from the definition of L and the
absoluteness of $\Delta_{0}$ formulas, just as in the classical case.
$\Delta_{0}$ Bounding follows from $\Pi_{2}$ Reflection. For
$\Pi_{2}$
Reflection, suppose {\bf M} $\models$ $\forall x \; \exists y \; \phi (x, y),
 \; \phi$ a $\Delta_{0}$ formula, and
A is a set in {\bf M}. We
will enlarge A to a transitive model of the same sentence. Work in the ambient
universe {\bf V}. All
of the parameters in sight (i.e. A and those of $\phi$) are members of
L$_{\omega_{1}^{CK}}$[G $\cup$ H$_{fin}$] for a fixed H$_{fin}$. Moreover, for
any
$\Delta_{0}$({\bf M}) formula $\phi$, there is an equivalent $\Delta_{0}$({\bf V}) formula
$\phi^{\ast}$, exactly as in \ref{IKP}. By standard symmetry
arguments, no more members of H are needed to find witnesses y in {\bf V}.
So for each x in
A let $\alpha_{x}$ be an ordinal such that a witness for x is in L$_{\alpha_{x}}$[G
$\cup$ H$_{fin}$].
Moreover, L$_{\alpha_{x}}$[G $\cup$ H$_{fin}$] should also include all the
relevant
evidence that the witness is in {\bf M}. Using the admissibility of the ambient
universe, let
$\alpha_{0}$ bound the $\alpha_{x}$'s needed. Of course, this must be iterated
to account
for the new sets in L$_{\alpha_{0}}$[G $\cup$ H$_{fin}$] which are in L$^{\bf
M}$ (via an
ordinal also in L$_{\alpha_{0}}$[G $\cup$ H$_{fin}$]). Iterate $\omega$-many
times. Let
$\alpha$ be $\bigcup \alpha_{n}$, and $\beta$ be \{x $\mid$ x is an ordinal in
{\bf
M}$_{big}\}$ over L$_{\alpha}$[G $\cup$ H$_{fin}$]. L$_{\beta}$ reflects $\phi$.
\end{proof}

\begin{lemma}
{\bf M} $\models$ $\neg \Sigma_{1}$DC.
\end{lemma}
\begin{proof}
As in \ref{not-sig}, there can be no sequence of distinct members of \{B(R)
$\mid R \in$
G\}, which
is a set in L.
\end{proof}

\begin{lemma}
{\bf M} $\models \neg$Resolvability.
\end{lemma}
\begin{proof}
This argument is somewhat technical, but since the idea is rather simple we will
give a quick overview of it. Suppose there were such a resolution. Since {\bf M} 
is definable in the ambient universe,
so is this resolution. Any definition can use only finitely many parameters, in
particular
only finitely many members of H. Consider branches of the form B(R), R $\in$ H.
Which
ordinal each gets attached to is determined by some forcing condition. Furthermore,
beyond the finitely many parameters, such a forcing
condition need mention only (that is, is an initial segment of) R. By mutual
genericity, for every initial segment of a real there are
infinitely many R's that share that initial segment. So any forcing condition attaching
B(R) to an ordinal forces infinitely many different B(R)'s to that same ordinal.
Therefore there is a set which contains infinitely many
B(R)'s, R $\in$ H, a contradiction.

The following proof makes essential use of the definability of the forcing relation
``p $\mid \vdash \phi$" over the ground model L$_{\omega_{1}^{CK}}$[G] for the forcing partial
order to get the set H. The forcing partial order is \{ p $\mid \; p : \omega \rightarrow
2^{<\omega},$ supp(p) finite\}, where supp(p), the support of p, is \{ n $\in \omega \mid
p(n) \not = \emptyset$\}. The resulting generic is (trivially identifiable with) a 
function f$_{H}$ with domain $\omega$ and range H. We will assume that the forcing
language has a constant H$_{n}$ for each $n \in \omega$, which is interpreted in the
generic extension as f$_{H}(n)$, as well as constants for f$_{H}$ itself and for 
the range of f$_{H}$, for which we
will ambiguously use the symbols f$_{H}$ and H respectively (f$_{H}$ and H being also 
sets in the  generic extension; in effect,
the constant H will be interpreted as the set H and similarly for f$_{H}$; whether
the constant symbol or its interpretation is meant every time we say ``H" of ``f$_{H}$"
should be clear from the context). To
name the elements in the generic extenion
L$_{\omega_{1}^{CK}}$[G, H] it suffices to consider the language
generated by the symbols H and the H$_{n}$'s; f$_{H}$ is not necessary. To name the 
elements in {\bf M} it suffices to consider the language generated by the H$_{n}$'s; 
H is not necessary. (Recall the definition of {\bf V}, over which {\bf M} was defined.)

{\bf M} is a structure definable over L$_{\omega_{1}^{CK}}$[G, H], and so the truth predicate
over {\bf M}, ``$\sigma \models \phi$", is also definable over L$_{\omega_{1}^{CK}}$[G, H].
It follows that the forcing relation p $\mid \vdash ``\sigma \models \phi"$ is definable 
over 
L$_{\omega_{1}^{CK}}$[G].
If the reader feels uncomfortable with the indirectness hidden behind the
simple notation ({\bf M} being a definable sub-class of {\bf M}$_{big}$, itself
a union not of sets but of definable classes), then to save the trouble of the
detailed argument needed to prove the definability of the forcing relation, 
with just a tad more set-theoretic power we can take the ground model to be 
L$_{\omega_{2}^{CK}}$[G], $\omega_{2}^{CK}$ the first admissible ordinal after
$\omega_{1}^{CK}$. (Our claim is, ultimately, that the results here are theorems
of ZF.) In L$_{\omega_{2}^{CK}}$[G, H] {\bf M} is a set, so the forcing 
relation  p $\mid \vdash ``\sigma \models \phi"$ for statements about {\bf M} 
is definable, $\Delta_{1}$ even. So take the ground model {\bf U} to be L$_{\omega_{1}^{CK}}$[G]
or L$_{\omega_{2}^{CK}}$[G], whichever you feel more comfortable with.

Suppose for a contradiction that Resolvability were to hold in {\bf M}:
$\sigma \models$ ``f is a resolution". Let p $\in$ f$_{H}$ force such:
p $\mid \vdash ``\sigma \models$ f is a resolution". Let supp(f) be \{ n $\mid
\;$ H$_{n}$ occurs in the definition of f\}. f mentions only finitely many paramters, 
each of which is ultimately evaluated 
in L$_{\omega_{1}^{CK}}$[G $\cup$ H$_{fin}$] for some finite H$_{fin} \subseteq$ H.
Hence supp(f) is finite, and without loss of generality supp(f) $\cup$ supp($\sigma$) 
$\subseteq$ supp(p). 

Not only does {\bf M} satisfy the Axiom of Foundation, it is even well-founded in the sense
of {\bf U}[H]; that is, the relation $(\sigma, x) <
(\tau, y)$ if ($\sigma \geq \tau \wedge \sigma \models ``x \in y"$) is 
well-founded in {\bf U}[H]. It would be shown inductively on $\gamma$ that each 
{\bf M}$_{\gamma}$ has this property in {\bf U}[H], hence {\bf M}$_{big}$ too as the union of
the {\bf M}$_{\gamma}$'s, and finally {\bf M} as a submodel of {\bf M}$_{big}$. Therefore,
not only do we have that $\sigma \models$ ``dom(f) is a well-founded linear order",
but also \{$\; \alpha \; \mid \; \sigma \models ``\alpha \in dom(f)"$\} (a collection in
{\bf U}[H] of Kripke sets) is also a well-founded linear order (with the order relation
being that induced by {\bf M}). We will refer to this set as ext$_{\sigma}$(dom f), the
externalization of the domain of f as $\sigma$.

So let q $\leq$ p, q $\in$ f$_{H}$, $\alpha$, and H$_{i}$ be such that 

q $\mid \vdash ``[\alpha$ is the least member of ext$_{\sigma}$(dom f)
such that for some X $\in$ H 

$\; 
\sigma \models (\bigwedge_{n \in supp(p)} B(H_{n}) \not = B(X) \; \wedge \; B(X) \in
f(\alpha)$)] 

$\wedge \; [\sigma \models (\bigwedge_{n \in supp(p)} B(H_{n}) \not = B(H_{i}) \; \wedge \;
B(H_{i}) \in f(\alpha))]$". 

Without loss of generality supp($\alpha$) $\cup \{i\} \subseteq$ supp(q). 
Consider j $\not \in$ supp(q). Let $\rho_{j}$ be the permutation of $\omega$ which
interchanges i and j. $\rho_{j}$ can be extended to a permutation of terms in the forcing 
language which interchanges all occurrences of H$_{i}$ with H$_{j}$. This permutation
we denote also as $\rho_{j}$. Such a permutation on terms induces a permutation on 
sentences (by merely permuting the parameters), and this permutation is truth-preserving
(for the language describing {\bf U}[H], i.e. without f$_{H}$):
$$r \mid \vdash \phi \; \leftrightarrow \; r \circ \rho_{j} \mid \vdash \rho_{j}(\phi).
$$
(For a development of focing and symmetric models, see \cite{JechST} for instance.) The
sentence forced by q above is in this restricted language, hence

q $\circ \rho_{j} \mid \vdash ``[\rho_{j}(\alpha)$ is the least member of 
ext$_{\sigma}$(dom f) such that 

for some X $\in$ H $\; \sigma \models (\bigwedge_{n \in supp(p)} B(H_{n}) \not = B(X) 
\; \wedge \; B(X) \in f(\rho_{j}(\alpha))$)] 

$\wedge \; [\sigma \models (\bigwedge_{n \in supp(p)} B(H_{n}) \not = B(H_{j}) \; \wedge \;
B(H_{j}) \in f(\rho_{j}(\alpha))]$". 

(Of course we are using the fact that $\rho_{j}(f) = f$.)

Let r $\leq$ q, q$\circ \rho_{j}$. Since r forces that both $\alpha$ and $\rho_{j}(\alpha)$
are the least {\bf M}-ordinals with a certain property, r $\mid \vdash \alpha = \rho_{j}(\alpha)$.
Hence 
r $\mid \vdash ``B(H_{i}), B(H_{j})
\in f(\alpha)"$. 

Let J $\subseteq \omega \backslash$ supp(q), $\mid$J$\mid$ = n finite. Let q $\circ 
\rho_{J}$ be \{ q $\circ \rho_{j} \mid$ j $\in$ J \} and r$_{J}$ be inf(q $\circ 
\rho_{J} \; \cup$ \{q\}). Then r$_{J} \mid \vdash ``B(H_{i}) \in f(\alpha)\; \wedge \; 
\bigwedge_{j \in J} B(H_{j}) \in f(\alpha)"$. The set R$_{n}$ = \{ r$_{J} \mid$ 
J $\subseteq \omega \backslash$ supp(q), $\mid$J$\mid$ = n \} is definable and pre-dense
in the forcing partial order, so f$_{H} \cap$ R$_{n} \not = \emptyset$. This means that 
f($\alpha$) contains at least n+1-many distinct Kripke sets of the form B(H$_{j}$). Since
this holds for all n, f($\alpha$) contains infinitely many such Kripke sets. This, though,
is a contradiction. According to the definition of a resolution, $\sigma \models$
``f($\alpha$) is a set", and each Kripke set in {\bf M} is in 
L$_{\omega_{1}^{CK}}$[G $\cup$ H$_{fin}$] for some finite H$_{fin}$. 
\end{proof}

\subsection{IKP + V=L + $\neg \Pi_{2}$-Reflection} \label{not Pi_2}

This construction is weaker than the others in that the model produced will
not satisfy $\neg \Pi_{2}$-Reflection, but merely will not satisfy
$\Pi_{2}$-Reflection. In fact it will be a model of $\neg \neg
\Pi_{2}$-Reflection. This suffices to show that IKP + V=L does not prove
$\Pi_{2}$-Reflection, but it is still open whether IKP + V=L proves $\neg \neg
\Pi_{2}$-Reflection.

The partial order of this model will be $\omega$-many copies of 2$^{<\omega}$
next to each other, with two consecutive bottom elements. That is, there is
a first node $\bot \bot$, followed by another node $\bot$, itself followed
by $\omega$-many incompatible nodes $\bot_{n}$, each of which is the bottom
node of a copy of 2$^{<\omega}$. We will really be working with the subtree
above $\bot$; the role of bottom is merely that $\bot \models ``1_{\bot \bot
} = 1$'', i.e. $\bot \models ``1 \in \hat{P}$''. Let $\hat{P}_{n}$ be the
nth subtree, internally: [$\sigma \models y \in$ $\hat{P}_{n}$]
$\leftrightarrow$
[$\exists \tau$ compatible with $\bot_{n} \;\; \sigma \models y = 1_{\tau}$];
notice that for $\sigma$ incompatible with $\bot_{n}$, $\sigma \models
\hat{P}_{n} = \{1\}$. (By way of notation, let the numbering of these trees
start with 1 instead of 0, because 0 as a subscript is already in use, to
indicate including 0 as a member.)

We need the $\omega$-sequence $< \hat{P}_{n} \mid n \geq 1>$ in the model.
To this end, let
$\alpha_{n}$ be $\hat{P}_{n} \; \cup \overline{n+1}$, where, as before,
$\overline{n+1}$
is the internalization of n+1. Notice that $\alpha_{n}$ is an ordinal, and
L$_{\alpha_{n}}$ = $\hat{P}_{n} \; \cup$ L$_{\overline{n+1}}$ (using lemma \ref{T-hat-0}
here). Let $\alpha$
be $\hat{P}$  $\cup
\; \overline{\omega} \cup
\{T_{n0} \mid n \geq 1\} \cup \{\alpha_{n} \mid n \geq 1\}$. So

$$
L_{\alpha} = \hat{P} \cup L_{\overline{\omega}} \cup \bigcup \{def(T_{n0})
\mid n \geq 1\} \cup
\bigcup \{def(\hat{P}_{n} \cup
L_{n+1}) \mid n \geq 1\}.$$
Over L$_{\alpha}$ let $\phi$(x, y) iff
\begin{eqnarray*}
0 \neq x \in \overline{\omega} &\wedge& 0 \neq y \subseteq 1 \\
&\wedge& \exists z (x, y \in z) \\
&\wedge& \forall z (x+1 \in z \wedge y \in z \rightarrow y = 1).
\end{eqnarray*}
We claim that $\forall x \; \bot \models ``\forall y \;
\phi(x, y) \leftrightarrow y \in \hat{P}_{x}$''. To see this, go through
the members of L$_{\alpha}$.
There is no z in $\hat{P}$ with a non-0 member. In L$_{\overline{\omega}}$
the only possible y (i.e. the only non-0 subset of 1) is 1 itself, and
indeed for all x $\geq$ 1 $\phi$(x, 1); since $\bot \models
1 \in \hat{P}_{x}$, this is fine. The only members of $\overline{\omega}$
also in $\hat{P}_{n0}$ are 0 and 1, so if
z $\in$ def($\hat{P}_{n0}$) has a legal x as a member then x = 1. In fact
this is subsumed by the
final case: the common members of $\overline{\omega}$ and ($\hat{P}_{n} \;
\cup$
L$_{\overline{n+1}}$) are 0, 1, ...
, $\overline{n}$; the non-0 subsets of 1 are exactly the members of T$_{n}$; so
the appropriate $<x, y>$
candidates are $<m, y>$, y $\in$ $\hat{P}_{n}$ and m$\leq$n. If m$<$n and
$\neg \sigma
\models y=1$ then \{m+1, y\} shows that $\neg \sigma \models \phi(m, y)$. Hence
the only
time $\sigma \models \phi(x, y)$ is when $\sigma \models y \in \hat{P}_{x}$.

Keeping in mind that $\hat{P}_{x} \in L_{\alpha}$, the functional relation
``f($\overline{n}$) = $\hat{P}_{\overline{n}}$'' (as a
two-place relation $\psi(\overline{n}, Z)$) can be defined over L$_{\alpha}$ as
``$\forall y \; y \in Z \leftrightarrow \phi(\overline{n}, y)$''. There is
a small integer m such that this function f is actually a
member of L$_{\alpha + \overline{m}}$. (L$_{\alpha}$ is available as a
parameter
already in L$_{\alpha + 1}$, and all else you need are some ordered pairs.)

Next comes an extension of the notion of a branch to include a branch through
$\hat{P}_{n}$. If
Q $\subseteq \hat{P}$ is closed downwards (i.e. closed under taking
supersets), then by  requiring in the definition of a branch that B
$\subseteq$ Q and restricting the $\gamma$'s
in clause (3) to Q, we get a branch through Q. The proof of lemma
\ref{externalization lemma} goes through with Q in place of $\hat{P}$. In
the case at hand, any branch through
$\hat{P}_{n}$, externally at some $\sigma$ compatible with $\bot_{n}$, is a
branch through $\hat{P}_{\sigma}$, and externally at some $\sigma$
incompatible with $\bot_{n}$ is just \{1\}. Let R be any real, and B$_{n}$
the corresponding branch through $\hat{P}_{n}$.

Let W be a recursive ordering of $\omega$ which is not well-founded but which
has no hyperarithmetic infinite descending sequence. For the convenience of the reader, 
we review some basic facts about W which will be used.

For every i there is a W-least j $\leq_{W}$ i such that the W-interval [j,i]$_{W}$
has order-type some $\alpha < \omega_{1}^{CK}$. To see this, work in 
L$_{\omega_{1}^{CK}}$.
Define f(0) = $\langle i,0,0 \rangle$, f(n+1) = the L-least triple $\langle j,\alpha,g
\rangle$ such that $j <_{W}$ proj$_{0}$(f(n)) and g is an order-isomorphism between
[j,i]$_{W}$ and the ordinal $\alpha$. Set $j_{n}$ = proj$_{0}$(f(n)). f is a 
$\Sigma_{1}$-definable function on an initial segment of $\omega$. By the admissibility 
of L$_{\omega_{1}^{CK}}$, f $\in$ L$_{\omega_{1}^{CK}}$. If dom f = $\omega$ then
i = $j_{0} >_{W} j_{1} >_{W} ...$ would be a hyperarithmetic infinite descending 
sequence, in contrast to the choice of W. Let n = max(dom f). Then [$j_{n}$,i]$_{W}$ is
isomorphic to some $\alpha \in \omega_{1}^{CK}$, but for all j $< j_{n}$ [j,i]$_{W}$
is not well-founded. Call this $j_{n}$ corresponding to a given i 0$_{i}$ (``0 in the
sense of i").

In the other direction, $\forall i \in \omega, \; \alpha \in \omega_{1}^{CK} \; \exists 
j_{\alpha} \in \omega$ such that $[i,j_{\alpha})_{W}$ has order-type $\alpha$ (except
possibly for those i's belonging to a final segment of W of order-type some $\beta <
\omega_{1}^{CK}$). To see this, for a given i, if this fails, let $\alpha$ be the least
ordinal for which this fails. By the admissibility of L$_{\omega_{1}^{CK}}$, 
$\{ j_{\beta} \mid \beta < \alpha \} \in$ L$_{\omega_{1}^{CK}}$. If this set is a final
segment of W, then the end-segment starting at 0$_{i}$ (as constructed above) is an
end-segment of order-type $\gamma$ = o.t.[0$_{i},i)_{W}$ + $\alpha < \omega_{1}^{CK}$.
In the other case, there is a j $>_{W}$ j$_{\beta}$ for all $\beta < \alpha$. In fact,
there is a least such j, or else the function f(0) = j $>$ j$_{\beta}$ (for all $\beta$
simultaneously) arbitrary, f(n+1) = the least (in the standard ordering of $\omega$) j
such that j $<_{W}$ f(n) and j $>$ j$_{\beta}$ ($\forall \beta$) would produce a 
hyperarithmetic infinite descending sequence. This least such j witnesses that [i,j) has
order-type $\alpha$, in contrast to the choice of $\alpha$. 

Combining these two results, we see that every i is contained in a maximal interval
of order-type an ordinal, and this interval starts at 0$_{i} \leq_{W}$ i and has order-type
$\omega_{1}^{CK}$ (or some $\beta < \omega_{1}^{CK}$ on a final segment). We will call
this interval i's standard neighborhood, with the notation sn(i). If $i_{1} >_{W} i_{0}$
is not in $i_{0}$'s standard neighborhood, then $\forall j \in$ sn($i_{0}$) $i_{1} >_{W} j$;
we will use the notation $i_{1} >_{W}$ sn($i_{0}$) in this case and say that $i_{1}$ is a 
non-standard distance above $i_{0}$. Notice that in this case there is an $\hat{i}$ such
that $i_{1} >_{W}$ sn($\hat{i}$) and $\hat{i} >_{W}$ sn($i_{0}$). To see this, notice that 
\{ $j \mid$ sn($i_{0}$) $<_{W} j <_{W} 0_{i_{1}}$ \} is non-empty, else sn($i_{0}$) would be a
member of L$_{\omega_{1}^{CK}}$ (definable as [$0_{i_{0}}, 0_{i_{1}}$)$_{W}$) and f(j) = 
order-type([$0_{i_{0}}, j$)$_{W}$) would then be a total function from sn($i_{0}$) onto
$\omega_{1}^{CK}$. Let $\hat{i}$ be any element such that sn($i_{0}) <_{W} \hat{i} <_{W} 0_{i_{1}}$.

With these preliminaries behind us, we can now return to the construction of the model
{\bf M}. Given a final segment
of W with least element i, let $\xi_{i}$ be \{B$_{j0} \mid j \geq_{W} i$\}
$\cup \; \alpha \cup \{\alpha\}$. As in section \ref{section-not-DC},
iterate definability $\omega_{1}^{CK}$-many times on top of L$_{\xi_{i}}$,
to get {\bf M}$_{i}$. As has already been observed,
lemmas \ref{IKP} - \ref{L} did not use any particular facts about the reals
used to define the branches,
and so in the following we will freely omit details of the proofs, referring the reader
to the proofs of these previous lemmas.

{\bf M} will be $\bigcup_{i \in I} {\bf M}_{i}$, for
some final segment I of W with no least element. The exact choice of I will
be given in lemma \ref{IKP2}, the only place where it's used.

\begin{lemma}
{\bf M}$_{i}$ $\models$ IKP. \label{local IKP}
\end{lemma}
\begin{proof}
As in lemma~\ref{IKP}.
\end{proof}

\begin{lemma}
L$_{\alpha} \in$ L$^{{\bf M}_{i}}$. \label{local L lemma}
\end{lemma}
\begin{proof}
$\alpha_{n}$ is definable over L$_{\alpha_{n}}$ as the set of all ordinals, so
$\alpha_{n}
\in$ L$_{\alpha}$ and $\alpha \subseteq$ L$_{\alpha}$. Since {\bf M}$_{i}$ is
admissible,
$\gamma$ = \{ $\beta \in L_{\alpha} \mid$ def(L$_{\beta}$) $\subseteq$
L$_{\alpha}$ \}
$\in$ {\bf M}$_{i}$. $\alpha \subseteq \gamma$, so L$_{\alpha}$ $\subseteq$
L$_{\gamma}$. By the definition of $\gamma$, L$_{\gamma}$ $\subseteq$
L$_{\alpha}$.
\end{proof}

\begin{corollary}
f $\in$ {\bf M}$_{i}$.
\end{corollary}
\begin{proof}
f $\in$ L$_{\alpha + \overline{n}}$, and since L$_{\gamma}$ = L$_{\alpha}$,
L$_{\gamma + \overline{n}}$ =
L$_{\alpha + \overline{n}}$.
\end{proof}

\begin{lemma}
In {\bf M}$_{i}$, if $\bot \models$ ``B is a branch through f(j)'' then j
$\geq_{W}$ i.
\end{lemma}
\begin{proof}
Suppose j $<_{W}$ i. Consider {\bf M}$_{i}$ at node $\bot_{j}$. For any
definition $\phi(x)$
at any stage in the construction of {\bf M}$_{i}$, consider the structure which
consists of all
of the relevant information: at first this includes the stage $\alpha$ at which
$\phi$ is being
evaluated, ({\bf M}$_{i}$)$_{\alpha}$, the formula $\phi$, and $\phi$'s
parameters, but these parameters can themselves be unpacked, until the only
parameters left are from $\hat{P}_{j}$. This uses of course that at
$\bot_{j}$ all of the branches included in $\xi_{j}$ are
\{1\}. The parameters are all of the form 1$_{\sigma}$. Let $\tau$ be any node
beyond or incompatible with all such $\sigma$'s. Then $\tau \models
\phi(1_{\rho})$ for some $\rho$ extending $\tau$ iff $\tau \models
\phi(1_{\rho})$ for every $\rho$ extending $\tau$ of the same length. (The
proof of this would use an automorphism of the model interchanging any
two such $\rho$'s.) So $\phi$ could not possibly define a branch through
$\hat{P}_{j}$.
\end{proof}

\begin{lemma}
If k $\leq_{W}$ i then L$_{\xi_{i}} \in$ L$^{{\bf M}_{k}}$.
\end{lemma}
\begin{proof}
In {\bf M}$_{k}$, let B$_{full}$ be \{ B $\in L_{\xi_{k}} \mid$ B is a branch
through f(j) for some j
$\geq_{W}$ i \}. Letting $\delta$ be \{ B$_{0} \mid$ B $\in B_{full}$ \} $\cup$
T, and letting
$\gamma$ be as in lemma \ref{local L lemma}, L$_{\xi_{i}}$ = L$_{\delta \cup
\gamma \cup
\{\gamma\}}$.
\end{proof}

\begin{corollary}
If k $\leq_{W}$ i then {\bf M}$_{i} \subseteq$ {\bf M}$_{k}$.
\end{corollary}
\begin{proof}
As in lemma \ref{resolvable}.
\end{proof}

\begin{corollary}
{\bf M}$_{i} \models$ V=L.
\end{corollary}
\begin{proof}
As in lemma \ref{L}.
\end{proof}

\begin{lemma}
{\bf M} $\models$ IKP. \label{IKP2}
\end{lemma}
\begin{proof}
As the union of transitive models of IKP (see lemma \ref{local IKP}), {\bf M} 
automatically satisfies Empty Set, Infinity, Union, and $\Delta_{0}$-Comprehension.
One of the more difficult axioms to check, surprisingly enough,
is Pairing. That is the point of the penultimate corollary, which also implies 
Extensionality. Foundation holds because {\bf M} is well-founded in {\bf V}.

The only other axiom is $\Delta_{0}$ Bounding. The choice of I will
be exactly to ensure this. The construction is really the same as in
\cite{metoo}; we replicate the main idea here to make this paper
self-contained. Notice that for any choice of I any node other
than $\bot$ forces IKP, because $\bot_{i} \models$ {\bf M} = {\bf M}$_{i}$ if i
$\in$ I, and $\bot_{i} \models$ {\bf M} = ``{\bf M}$_{i}$ without the
branch'' which also models IKP, if  i $\not\in$ I. So we need concern
ourselves only with functions forced to be total by $\bot$. There are
countably many $\Delta_{0}$ definitions (with parameters) which might
define a total function on a set. We'll handle the nth, $\phi_{n}(x, y)$, at
stage n. Assume
inductively we have max$_{n}$ a non-standard distance above min$_{n}$ in
W (max$_{0}$ and min$_{0}$ being arbitrary elements of W a non-standard distance apart).
By the discussion about W (just before the lemmas of this section), let m be a
non-standard distance above min$_{n}$ and a non-standard distance below max$_{n}$ (i.e.
sn(min$_{n}$) $<$ m and sn(m) $<$ max$_{n}$). In {\bf M$_{m}$} either $\phi_{n}$ defines
a total function or it doesn't. In the former case, {\bf M$_{m}$} contains a bounding
set (by lemma \ref{local IKP}), which is also a bounding set in any model extending
{\bf M$_{m}$}, so let max$_{n+1}$ = m (and let min$_{n+1}$ = min$_{n}$). In the latter
case, $\phi_{n}$ is not total in {\bf M$_{m}$}, or, for that matter, in any transitive 
submodel thereof (by the absoluteness of $\Delta_{0}$ formulas among transitive models).
So let min$_{n+1}$ = m (and let
max$_{n+1}$ = max$_{n}$). Let I = \{i $\mid \exists$n i $\geq$ max$_{n}$\}.
\end{proof}

\begin{lemma}
{\bf M} $\models$ V=L.
\end{lemma}
\begin{proof}
The union of models of V=L is also a model of V=L.
\end{proof}

\begin{lemma}
{\bf M} $\not \models \Pi_{2}$-Reflection.
\end{lemma}
\begin{proof}
The assertion ``B is a branch through S'' is $\Delta_{0}$, hence absolute
in all
transitive sets containing the parameters. So if $\bot \models$ ``B is a
branch through f(j)'' then j $\in$ I. Moreover, if j $\in$ I, then $\bot
\models$ ``there is a branch through f(j)''. So $\bot \models ``\forall j,
B_{j} \; \exists i, B_{i}$ if B$_{j}$ is a branch through f(j) then either
f(j) = \{1\} or (i$<_{W}$j and B$_{i}$  is a branch through f(i))''. But
this does not reflect to any set, because each set is a member of some {\bf
M}$_{i}$, which is
constructible in L$_{\omega_{1}^{CK}}$, in which W has no infinite
descending sequence.
\end{proof}

\section{Questions}

$\;\;\;\;\;$1. Is there a model of $\Sigma_{1}$ DC + IKP + V=L +
$\neg$Resolvability?

2. Is there a model of IKP + V=L + $\neg \Pi_{2}$-Reflection?

3. Does IZF prove that admissible sets are unbounded (that is, for each set
X is
there a
transitive admissible set containing X)? Does IZF prove the existence of an
admissible set
with $\omega$? The usual constructions of admissible sets use either reflection
or the
linearity of the ordinals, neither of which is available to us here.

4. Does IZF$_{Rep}$ (that is, IZF with the Bounding Schema substituted by the
Replacement Schema) prove IKP? With some better technology (see question 6) it
should be possible to show that
IZF$_{Rep}$ - Power does not prove IKP. Also, \cite{FS} (with a minor
alteration) shows
that IZF$_{Rep}$ does not prove $\Pi_{1}$ Bounding. But whether IZF$_{Rep}$
proves
$\Delta_{0}$ Bounding is apparently unknown.

5. As has already been observed, least fixed points and greatest fixed points
are closely
connected classically, but not so intuitionistically. How much can they be
separated? Can
IKP be strengthened so that even more least fixed points become definable,
without
picking up $\Pi$ Persistence or any greatest fixed points? Conversely, can
$\Pi$
Persistence by extended to get more greatest fixed points, without picking up
IKP or least
fixed points? There are connections between least and greatest fixed points and
their
categorical analogues, initial algebras and final co-algebras of functors, so
these matters
could be pursued from this angle; but not so much is known even there, and the
connections are not yet so clear. These considerations suggest a similar yet
simpler
question, involving the first order quantifiers: has anyone investigated
existential
intuitionistic logic (with $\exists$, even $\neg \exists \neg$, but not
$\forall$), or universal
intuitionistic logic?

6. Are there other natural examples of the difference between IKP and $\Pi$
Persistence?

7. Another variant on IKP yet to be mentioned is the weakening of $\Delta_{0}$
Bounding
to $\Sigma_{1}$ Replacement. It can be shown that classically $\Sigma_{1}$
Replacement
does not imply $\Delta_{0}$ Bounding; in fact, it is possible to build a
classical
model of full Replacement + $\neg \Delta_{0}$ Bounding. It would have been nice
to have included in this paper an adaptation of this model to IL. However, this
project brings up a host of new problems. All of the models presented here
which use forcing do so only in
an ambient classical universe to get some particular desired property. For
instance, a failure of $\Sigma_{1}$ DC in a classical model is imported to
become a failure
of $\Sigma_{1}$ DC in the intuitionistic model. In order to get $\Delta_{0}$
Bounding to fail, the forcing has to be done internally. Unfortunately,
intuitionistic forcing technology is not yet well enough developed to carry
such a burden. Certainly some work has already been done in this area; see
\cite{beeson} for instance. But not enough has been done yet to provide a
robust enough theory to carry through the construction envisioned here.
This should be possible, but would take us too far beyond the scope of the
current paper.

\end{document}